\documentclass[12pt,a4paper,oneside]{amsart}

\ifdefined\SMART
\usepackage[paperwidth=12cm,paperheight=24cm,left=5mm,right=5mm,top=10mm,bottom=5mm]{geometry}
\else
\calclayout
\fi

\usepackage{amssymb,latexsym,amsfonts,textcomp,fancyhdr,calc,graphicx}
\usepackage{pdfpages,amsfonts,amsthm,amssymb,latexsym,graphicx,leftidx,fancyhdr}
\usepackage{amsmath,amstext,amsthm,amssymb,amsxtra}
\usepackage[allcolors=blue,allbordercolors=blue,pdfborderstyle={/S/U/W 1}]{hyperref}
\usepackage[ocgcolorlinks]{ocgx2}
\usepackage{dsfont}
\usepackage[framemethod=tikz]{mdframed}
\usepackage{asymptote}
\usepackage{enumerate}

\usepackage{txfonts,pxfonts,tikz} 
\usepackage{tgschola}
\usepackage[T1]{fontenc}
\usepackage{mathtools}

\newtheorem{theorem}{Theorem}[section]
\newtheorem{corollary}{Corollary}[section]
\newtheorem{lemma}{Lemma}[section]

\newtheorem{question}{Question}

\theoremstyle{definition}

\newenvironment{que}
    {\begin{quotation}\vskip -1.5em \begin{question}
    }
    { 
    \end{question}\end{quotation}
    }

\newcommand{\beql}[1]{\begin{equation}\label{#1}}
\newcommand{\eeq}{\end{equation}}
\newcommand{\comment}[1]{}
\newcommand{\Ds}{\displaystyle}

\newcommand{\Abs}[1]{{\left|{#1}\right|}}

\newcommand{\Linf}[1]{{\left\|{#1}\right\|_\infty}}
\newcommand{\Norm}[1]{{\left\|{#1}\right\|}}

\newcommand{\Floor}[1]{{\left\lfloor{#1}\right\rfloor}}

\newcommand{\Set}[1]{{\left\{{#1}\right\}}}

\newcommand{\RR}{{\mathbb R}}

\newcommand{\CC}{{\mathbb C}}
\newcommand{\ZZ}{{\mathbb Z}}

\newcommand{\TT}{{\mathbb T}}
\newcommand{\QQ}{{\mathbb Q}}

\newcommand{\one}{{\bf 1}}

\newcommand{\inner}[2]{{\langle #1, #2 \rangle}}

\newcommand{\dens}{{\rm dens\,}}
\newcommand{\supp}{{\rm supp\,}}

\newcommand{\vol}{{\rm vol\,}}
\newcommand{\ft}[1]{\widehat{#1}}

\newcommand{\qq}[1]{{{\begin{que} {#1} \end{que}}}}
\newcounter{rem}
\newcounter{step}
\newcounter{mysec}
\newcounter{mysubsec}[mysec]
\usepackage{soul}

\title{Orthogonal Fourier Analysis on domains}

\author{Mihail N. Kolountzakis}
\address{\href{http://math.uoc.gr/en/index.html}{Department of Mathematics and Applied Mathematics}, University of Crete,\\Voutes Campus, 70013 Heraklion, Greece,\\and\\ \href{https://ics.forth.gr/}{Institute of Computer Science}, Foundation of Research and Technology Hellas, N. Plastira 100, Vassilika Vouton, 700 13, Heraklion, Greece}
\email{kolount@gmail.com}

\date{\today}

\begin{document}

\begin{abstract}
In this paper we go over the history of the Fuglede or Spectral Set Conjecture as it has developed over the last 30 years or so. We do not aim to be exhaustive and we do not cover important areas of development such as the results on the problem in classes of finite groups or the version of the problem that focuses on spectral measures instead of sets. The selection of the material has been strongly influenced by personal taste, history and capabilities. We are trying to be more descriptive than detailed and we point out several open questions.  
\end{abstract}




\maketitle

\begin{center}
\textit{For Bent Fuglede. Who started all this.}
\end{center}

\tableofcontents

\section{Fourier Analysis on domains}\label{s:fa-on-domains}


\subsection{Sets with orthogonal bases of exponentials}\label{ss:ortho-bases}

Fourier Analysis allows us to decompose a function on a (locally compact) abelian group, say $\RR^d$, $\TT=\RR/\ZZ$, or $\ZZ_N = \ZZ/(N\ZZ)$, as a linear combination of characters of the group. For instance, when we decompose functions on the torus $\TT$ (which we may view as the interval $[0, 1)$) each function is written as a Fourier series
\begin{equation}\label{fourier-series}
f(x) = \sum_{n=-\infty}^\infty \ft{f}(n) e^{2\pi i n x}.
\end{equation}
The characters here (continuous homomorphisms from the group to the multiplicative group $\CC$) are the functions $e^{2\pi i n x}$ defined for $x \in \TT$ and indexed by $n \in \ZZ$. Thus we view $\ZZ$ as the \textit{dual group} of $\TT$, the group of the characters (or frequencies) on $\TT$. Of course, one must be more precise about which functions on $\TT$ are to be expanded and how the series is understood to converge, and classical Fourier Analysis is all about subtle questions arising from this expansion \cite{zygmund2002trigonometric,katznelson2004introduction,wolff2003lectures}, questions that have driven the development of Mathematical Analysis and other branches of Mathematics from 1900 and even earlier.

But the most useful and best understood such expansion is when $f$ is in the Hilbert space $L^2(\TT)$ and the series is understood as convergent in the $L^2$ norm. In this case we have the huge advantage that the characters $e^{2\pi i n x}$ are orthogonal, and all the conveniences that arise from orthogonal expansions in Hilbert space, such as Parseval's formula $\Norm{f}_2^2 = \sum_{n \in  \ZZ} \Abs{\ft{f}(n)}^2$. This is by far the part of Fourier Analysis most used in applications, either in pure or applied Mathematics or in other sciences. And orthogonality is critical.

What happens though when we restrict our functions' support to be a subset of our group? To stay in the context of the torus $\TT$, assume, for example, that we have a function $f \in L^2([0, \frac12])$ and want to expand it as in \eqref{fourier-series}. Of course we can extend $f$ to be zero on the rest of $\TT$ and use the original expansion \eqref{fourier-series} in $\TT$, but then the constituent parts of $f$, the characters $e^{2\pi i n x}$ are no longer orthogonal in $L^2([0, \frac12])$. What is more, by doing this extension by zero to the rest of $\TT$, we end up using a lot more frequencies (characters) than we actually need (exactly twice more in this case, in a very well defined sense). This is not unexpected since this extension by zero to the whole group is certainly wasteful. Clearly the ideal situation here would be to have an expansion like \eqref{fourier-series} in which the summands are orthogonal in our domain, $[0, \frac12]$.

In this particular case we are lucky. It is enough to take all even frequencies $n \in 2\ZZ$ in order to be able to orthogonally expand every $f \in L^2([0, \frac12])$ and in a unique way. Writing $e_\lambda(x) = e^{2\pi i \lambda \cdot x}$ and observing that $\Norm{e_\lambda}_{L^2([0, \frac12]} = 2^{-1/2}$ we have
$$
f(x) = \sum_{n \in 2\ZZ} \inner{f}{\sqrt2\, e_{2n}} e_{2n}(x).
$$
With half as many frequencies as before we have our orthogonal expansion. Let us call the set of frequencies we used
$$
\Lambda = \Set{2n:\ n\in\ZZ}
$$
a \textit{spectrum} of the set $E = [0, \frac12]$ (see Fig.\ \ref{fig:half-interval}).

It is precisely this question that we deal with in this paper.
\begin{quotation}
{\bf Main Question:}
If $E$ is a subset of the locally compact abelian group $G$ of finite Haar measure, when can we find a set of characters $\Lambda \subseteq \ft{G}$ (here $\ft{G}$ is the dual group of $G$ \cite{rudin1962groups}, the group of continuous characters on $G$) such that $\Lambda$ is orthogonal on $E$ and complete on $L^2(E)$?

If we can find such a $\Lambda$ we call it a \textit{spectrum} of $E$ and $E$ itself is called a \textit{spectral set}.
\end{quotation}

\begin{figure}[h]
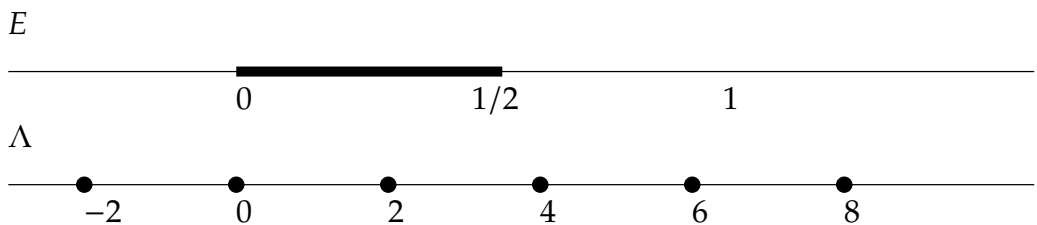
\caption{The set $E = [0, 1/2]$ and its spectrum $\Lambda = 2\ZZ$. This set $2\ZZ$ is not the only spectrum of $E$. Any translate of it is also a spectrum and this is generally true in all groups.}\label{fig:half-interval}
\end{figure}

It should be clear that the spectrum $\Lambda$ of $E$, when it exists, is generally not unique: any translate (in $\ft{G}$) of $\Lambda$ is again a spectrum of $E$.

Let us now change the group and work in $\RR$ instead of $\TT$ (note: $\TT$ is not a subgroup of $\RR$!). Consider the unit interval $I = [0, 1] \subseteq \RR$. By \eqref{fourier-series} we immediately get that $\ZZ \subseteq \ft{\RR} \simeq \RR$ is a spectrum of $I$, and so is any translate of $\ZZ$. Similarly viewing $E = [0, 1/2]$ as a subset of $\RR$ now we again get that $2\ZZ$ and any of its translates are spectra of $E$. And in the group $\RR^d$ the unit cube $I_d = [0, 1]^d$ is again a spectral set one of whose spectra is $\ZZ^d$ (this is the $L^2$ theory of multivariable Fourier series). But $I_d$ has many more spectra than the translates of $\ZZ^d$ for $d \ge 2$ \cite{iosevich1998spectral,lagarias2000orthonormal,kolountzakis2000packing} or see \S \ref{ss:spectra-of-cube}. One spectrum of the unit square $I_2$ is shown in Fig.\ \ref{fig:square} which is not a translate of $\ZZ^2$.

\begin{figure}[h]
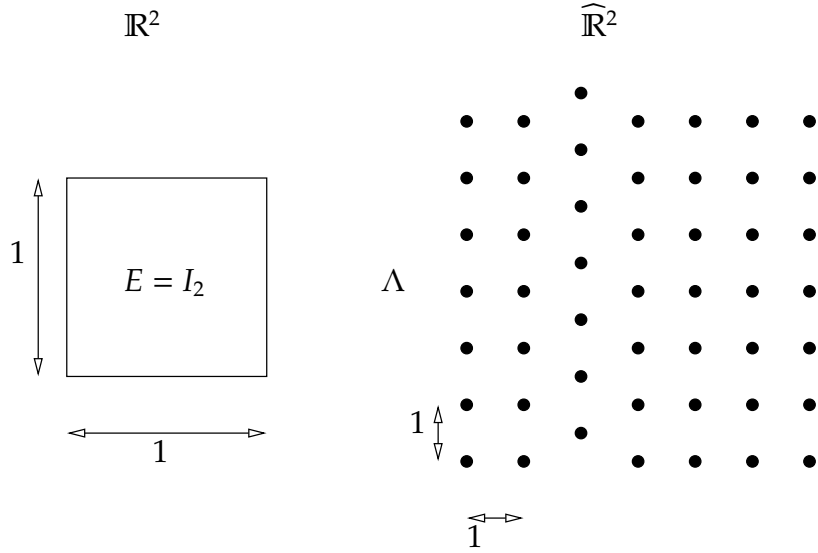
\caption{The unit square $I_2$ in the plane and one of its spectra, which consists of $\ZZ^2$ but with one of the columns shifted arbitrarily in the vertical direction. There is a complete description of all spectra of $I_2$ in \S \ref{ss:spectra-of-cube}.}\label{fig:square}
\end{figure}

A more interesting example of a spectral set in $\RR$ and its spectrum (again in $\RR$) is the set $E = [0, 1/2] \cup [1, 3/2]$ one of whose spectra is the set $\Lambda = 2\ZZ \cup (2\ZZ-\frac12)$ (see Fig.\ \ref{fig:two-intervals}).

\begin{figure}[h]
\ifdefined\SMART\resizebox{10cm}{!}{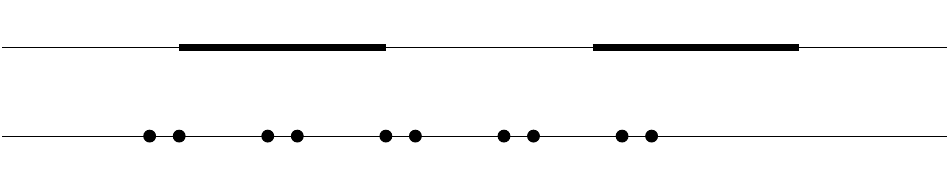}\else
\input two-intervals.pdf_t
\fi
\caption{The union of two intervals $[0, 1/2] \cup [1, 3/2]$ is a spectral set in $\RR$ and $\Lambda = 2\ZZ \cup (2\ZZ-\frac12)$ is its spectrum.}\label{fig:two-intervals}
\end{figure}

\subsection{The Fuglede or Spectral Set conjecture}\label{ss:fuglede-ss-conjecture}

It was Fuglede \cite{fuglede1974operators} who first posed the question of spectrality posing a conjecture.

\begin{quotation}
{\bf Fuglede conjecture or Spectral Set Conjecture}:
A set $E \subseteq \RR^d$ is spectral if and only if $E$ can tile $\RR^d$ by translations.
\end{quotation}

For $E$ to be able to \textit{tile $\RR^d$ by translations} means that it is possible to translate $E$ to some locations $T\subseteq \RR^d$ so that almost all points in $\RR^d$ (in the sense of Lebesgue measure) belong exactly to one $T$-translate of $E$.

The reason Fuglede was interested in spectral sets is that he had proved spectrality of a set $E \subseteq \RR^d$ to be equivalent to the possibility of restricting to $L^2(E)$ the distributional-sense partial differentiations $-i \partial_{x_1}, \ldots, -i \partial_{x_d}$ so that they become commuting self-adjoint operators on $L^2(E)$ (this was a question of Segal).

Fuglede himself proved in \cite{fuglede1974operators} that if one adds the word \textit{lattice} to both sides of the conjecture then it becomes true:

\begin{theorem}[\cite{fuglede1974operators}]\label{thm:lattice-fuglede}
Suppose $E \subseteq \RR^d$ has finite measure. Then $E$ tiles $\RR^d$ when translated by the lattice $L$ if and only if the dual lattice $L^*$ of $L$ is a spectrum of $E$.
\end{theorem}

A lattice $L$ in $\RR^d$ is a subgroup of $\RR^d$ generated by $d$ linearly independent vectors. In other words $L = A\ZZ^d$ where the $d\times d$ matrix $A$ is non-singular. Then $L^*$, the dual lattice of $L$, is the lattice $A^{-\top}\ZZ^d$. So $(L^*)^* = L$.

We will see the proof of Theorem \ref{thm:lattice-fuglede} in \S \ref{ss:lattice-fuglede}.

Fuglede also proved that the conjecture is true in the case of a triangle or a disk in the plane: since they clearly do not tile by translations, they are also not spectral. (See also \cite{fuglede-ball,iosevich2001convexbodies,kolountzakis2004study,kolountzakis2004distance} for the case of the ball.)
In Fig.\ \ref{fig:some-domains} we show some spectral and some non-spectral domains which satisfy the Spectral Set Conjecture.

\begin{figure}[h]
\ifdefined\SMART\resizebox{10cm}{!}{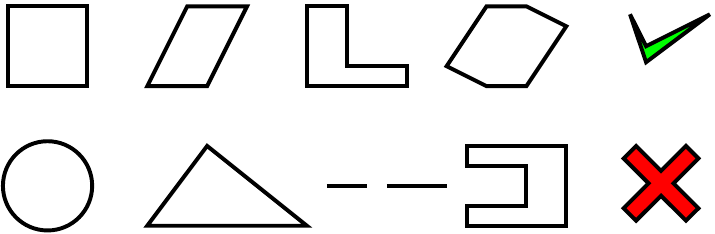}\else
\input various.pdf_t
\fi
\caption{Some domains in $\RR^2$ and $\RR$ which are and are not spectral.}\label{fig:some-domains}
\end{figure}

Theorem \ref{thm:lattice-fuglede} immediately furnishes us with further examples of spectral sets: all \textit{fundamental domains} of lattices have a lattice spectrum. A lattice is a subgroup of $\RR^d$ and any selection of one element from each of its cosets makes up a \textit{fundamental domain} of a lattice. For the lattice $L = A \ZZ^d$ the fundamental parallelepiped $A[0, 1)^d$ is such a fundamental domain, but we can construct many others as follows: start with a fundamental domain $\Omega$, cut off a piece of it and move it by an element of $L$. This motion does not change the coset where each moved element of $\Omega$ belongs, so it remains a fundamental domain. A fundamental domain does not have to be bounded as this process can be repeated infinitely often. So, a fundamental domain $E$ of a lattice $L$, is exactly what we call a \textit{lattice tile}: the translates $E+\ell$, $\ell \in L$, are such that \textit{almost all} points in $\RR^d$ are covered by exactly one of them. We emphasize that we only demand this exact covering almost everywhere, so a fundamental domain, for us, can be altered on a set of measure zero and it still remains a fundamental domain.

\begin{figure}[h]
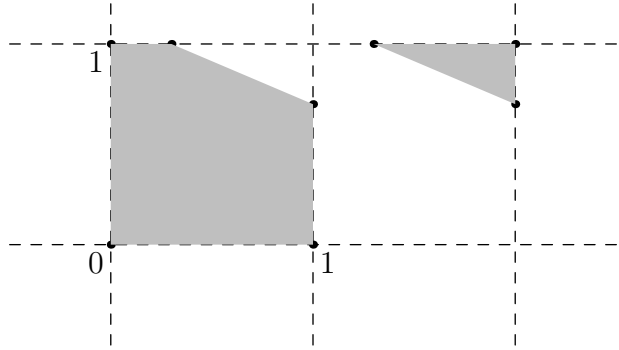

\begin{center}
\begin{asy}
size(8cm);

pair a=(0,0), b=(1,0), c=(1,0.7), d=(0.3,1), e=(0,1), c1=(2,0.7), d1=(1.3,1), f=(2,1);

draw((-0.5,0) -- (2.5,0), dashed);
draw((-0.5,1) -- (2.5,1), dashed);
draw((0,-0.5)-- (0,1.2), dashed);
draw((1,-0.5)-- (1,1.2), dashed);
draw((2,-0.5)-- (2,1.2), dashed);

dot(a); dot(b); dot(c); dot(d); dot(e); dot(c1); dot(d1); dot(f);

fill(a -- b -- c -- d -- e -- cycle, mediumgray);
fill(c1 -- f -- d1 -- cycle, mediumgray);
label("0", (0,0), SW); label("1", (1,0), SE); label("1",(0,1), SW);
//label("$\Omega$", (0.5,0.5), SW);
//label("$\Omega$", (1./3.)*(c1+f+d1), E);

\end{asy}

\caption{Shaded region is a fundamental domain of $\ZZ^2$ in $\RR^2$. It arises by cutting off a piece of the unit square and moving it by an integer vector.}\label{fig:fd}
\end{center}
\end{figure}

Orthogonality of the exponentials $e_\lambda(x)$ and $e_\mu(x)$ on $L^2(E)$ is easy to describe. Since
$$
\inner{e_\lambda}{e_\mu} = \int_E  e_\lambda(x) \overline{e_\mu(x)}\,dx =
\int \one_E(x) e^{2\pi i (\lambda-\mu)x}\,dx
$$
it follows that the frequencies $\lambda$ and $\mu$ are orthogonal on $E$ (meaning $e_\lambda$ and $e_\mu$ are orthogonal in $L^2(E)$) if and only if
\begin{equation}\label{diff-in-zeros}
\ft{\one_E}(\mu-\lambda) = 0.
\end{equation}
Here $\ft{\one_E}$ is the Fourier Transform of $\one_E$, the indicator function of $E$. In general our Fourier transform on $\RR^d$ is defined with the normalization
$$
\ft{f}(t) = \int_{\RR^d} f(x) e^{-2\pi i t \cdot x}\,dx,
$$
for $f \in L^1(\RR^d)$. 

\subsection{What this survey is not about}\label{ss:not-about}

In this paper, which is some sort of continuation of \cite{kolountzakis2004study}, I have emphasized the areas of the problem that I am most familiar with and that I have worked on most. It is inevitable that some areas are neglected.

The most important omission is the part of the theory that has been developed, and is still developing ever more intensely, around finite abelian groups. As explained below in \S \ref{ss:spectrality-in-groups} the notions of tiling and spectrality make perfect sense in all locally compact abelian groups. This is not a generalization for its own sake. A very big part of the development of the Fuglede conjecture, including its disproofs presented in \S \ref{ss:st-fails} and \S \ref{ss:ts-fails}, are first done in finite abelian groups, then are pulled on to $\ZZ^d$ and $\RR^d$. There are many established reductions of both directions of the Fuglede Conjecture from one group to another \cite{dutkay2014some}.

A lot of work has been done on classes of cyclic groups $\ZZ_N$, usually restricted by how many prime factors are allowed into $N$ and to what exponents \cite{laba2024splitting,laba2022splitting,laba2023coven,laba2022combinatorial,malikiosis2022structure,malikiosis2017fuglede,kiss2022fuglede,kiss2020discrete,laba2002spectral,zhang2024group} with many of them based on the influential paper \cite{coven1999tiling}. There are also several results concerning products of cyclic groups with few factors \cite{iosevich2017fuglede,fallon2019fuglede,fallon2023spectral,fallon2022spectral,kiss2021fuglede,malikiosis2024linear,zhang2023fuglede,aten2017tiling,ferguson2020fuglede,kiss2024tiling,mattheus2020counterexample,shi2020equi}

The other major omission of this survey is the work that has been done on a natural generalization of spectral sets: \textit{spectral measures}. Given a measure $\mu$ on a locally compact abelian group $G$, when can we find a collection of characters $\Lambda \subseteq \ft{G}$ that form an orthogonal basis of $L^2(\mu)$. Connections with tiling are weaker here and the main question is which measures are spectral or not, meaning for which measures such a set $\Lambda$ of characters exists that forms an orthogonal basis of $L^2(\mu)$.

It was first pointed out in \cite{jorgensen1998dense,strichartz2000mock} who indicated several examples of Cantor-type sets which are spectral and which are not. It was shown, for example, that the usual ternary Cantor set with its natural measure (start with Lebesgue measure on $[0, 1]$ and each time you throw out an interval from the middle of another redistribute its measure equally to its two neighbors) is not spectral, while a variant (at each stage split each interval intro 4 equal intervals then keep only the first and third) is spectral. There has been a huge number of papers since then such as \cite{laba2002spectral-cantor,dutkay2019hadamard,dai2012does,an2019spectral,dai2013spectral}.

When one leaves orthogonality of the exponentials behind, but still requires a collection of exponentials that form a basis of some sort (a Riesz basis or a frame, typically) the problem's nature changes completely (and in some sense it becomes more interesting as it concerns many more domains). It loses, to a great but not complete extent, the algebraic or number theoretic character which arises from the identities that guarantee the orthogonality and becomes a more quantitative subject, whose results do not resemble much those in the Spectral Set problem. Still there are some borderline similarities as evidenced by some papers \cite{grepstad2014multi,kolountzakis2015multiple,lai2023riesz,debernardi2022riesz}. But for the most part it is a different subject. It took decades, for example, to find a bounded measurable set in $\RR$ of positive measure that does not have a Riesz basis of exponentials \cite{kozma2023set}. It was not much easier to show that finite unions of intervals do have a Riesz basis of exponentials \cite{kozma2015combining,kozma2016combining} or to prove that unbounded sets of finite measure in $\RR$ have a frame of exponentials \cite{nitzan2016exponential}. Other interesting results concern measures that are mixtures of different dimensions or linear but embedded in higher dimension \cite{lev2018fourier,lai2021spectral,iosevich2022fourier}. When one restricts the frequencies or demands that, say, a basis for a union is a union of the individual bases new interesting phenomena arise \cite{pfander2019riesz,pfander2024exponential,lee2024exponential}.

\section{Tiling by translation}\label{s:tiling}


\subsection{Tiling by a function}\label{ss:tiling-by-a-function}

So far we have dealt with a set $E$ tiling $\RR^d$, say, by translations at a set $T \subseteq \RR^d$. Since we are only going to deal with sets $E$ of positive measure it follows that $T$ is a countable, discrete set. An easy way to define translational tiling by $E$ is to demand that
$$
\sum_{t \in T} \one_E(x-t) = 1,\ \ \text{ for almost every } x \in \RR^d.
$$
This description of tiling makes the notion of tiling more amenable to manipulation and analysis.

Once we write this it begs for generalization. Let $f:\RR^d\to\RR$ be a measurable function (in most cases we will deal with functions in $L^1(\RR^d)$, often nonnegative). We say it tiles at level $\ell \in \RR$ when translated at $T$ if
\begin{equation}\label{soft}
\sum_{t \in T} f(x-t) = \ell,
\end{equation}
for almost all $x \in \RR^d$ and with the sum converging absolutely. We do allow $T$ to be a multiset: some $t \in T$ may appear more than once. See Fig.\ \ref{fig:soft-tile} for an example.

\begin{figure}[h]
\ifdefined\SMART\resizebox{10cm}{!}{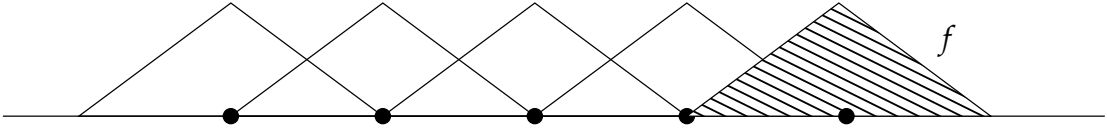}\else
\input soft-tile.pdf_t
\fi
\caption{The function $f$ with the shaded triangle graph tiles $\RR$ when translated at the locations shown by black dots. The level of the tiling is the height of the triangle.}\label{fig:soft-tile}
\end{figure}

It makes sense to restrict $T$ to be of bounded density (the ratio $\Abs{T \cap Q} / \vol(Q)$ is bounded, for all cubes $Q \subseteq \RR^d$ of sidelength $\ge 1$), a condition that is automatic in several natural cases, the most important of which is the case $f \ge 0$ and $0 < \int f < \infty$ \cite{kolountzakis1996structure}.

By the bounded density of $T$ it follows that the measure
$$
\delta_T = \sum_{t \in T} \delta_t
$$
which consists of one unit point mass at each point of $T$ (or, if $T$ is a multiset, a point mass at $t$ equal to the multiplicity of $t \in T$) is a \textit{tempered distribution} \cite{strichartz2003guide}. Informally speaking (see \cite{kolountzakis1996structure} for all the details) we can rewrite \eqref{soft} as the convolution
\begin{equation}\label{tiling-as-convolution}
f * \delta_T = \ell.
\end{equation}

\subsection{Tiling seen on the Fourier side}\label{ss:fourier-tiling}

Taking the Fourier Transform of \eqref{tiling-as-convolution} we obtain, at least formally (but see \cite{kolountzakis1996structure,kolountzakis2016non} for a more rigorous discussion),
\begin{equation}\label{ft-product}
\ft{f} \cdot \ft{\delta_T} = \ell \delta_0.
\end{equation}
Let us remark here that $\ft{f}$ is a continuous function on $\RR^d$ (being the Fourier Transform of an integrable function) but $\ft{\delta_T}$ is, again, a tempered distribution, and not necessarily locally a measure.

Comparing the supports of the two sides of \eqref{ft-product} we see that the supports of $\ft{f}$ and $\ft{\delta_T}$ are disjoint \cite{kolountzakis1996structure} apart from the origin. Wherever $\ft{\delta_T}$ lives $\ft{f}$ must kill it, except at the origin. Thus the following is a \textit{necessary} condition for tiling \eqref{soft} to happen:
\begin{equation}\label{supports}
\supp \ft{\delta_T} \subseteq \Set{\xi\in\RR^d: \ft{f}(\xi) = 0} \cup \Set{0}.
\end{equation}

\begin{figure}[h]
\ifdefined\SMART\resizebox{10cm}{!}{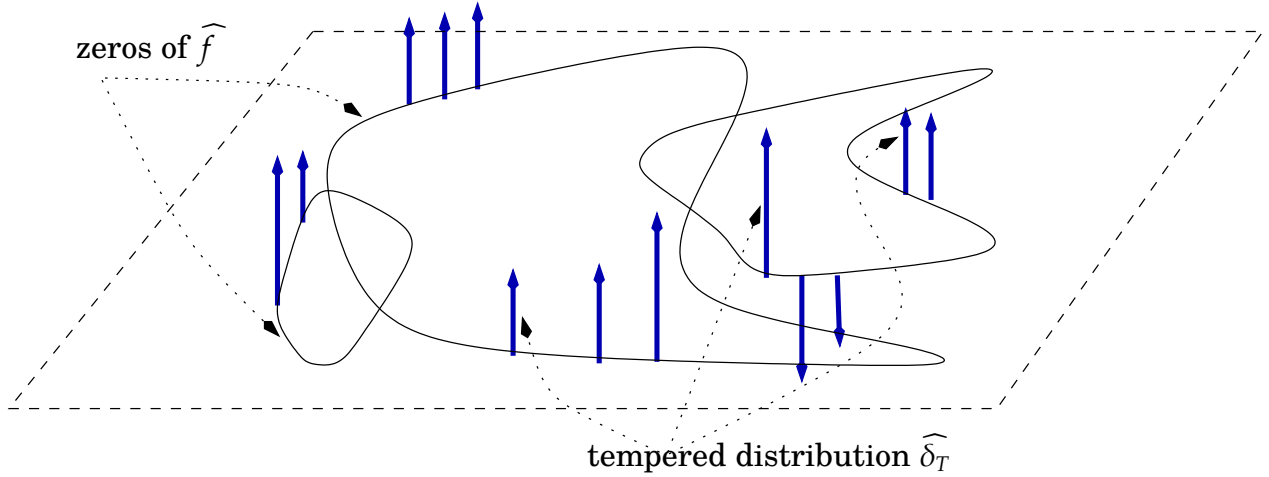}\else
\input supports.pdf_t
\fi
\caption{The sero set of $\ft{f}$, denoted by $Z(f)$, supports the tempered distribution $\ft{\delta_T}$, the Fourier Transform of the measure $\delta_T$ which encodes the translation multiset.}
\label{fig:supports}
\end{figure}

Condition \eqref{supports} is almost a sufficient condition for tiling as well. If we somehow know that $\ft{\delta_T}$ is locally a measure then \eqref{supports} implies \eqref{soft}: roughly any zero of $\ft{f}$ is enough to kill a measure, but if $\ft{f}$ has to kill a higher order tempered distribution then its zeros must be deeper. The textbook example of this situation is that $x\cdot \delta_0'$ is not the 0 distribution, but $x^2 \cdot \delta_0' = 0$. Here $\delta_0'$ is the derivative of the measure $\delta_0$, which is a distribution of first order, not locally a measure. The simple zero of $x$ at $0$ is not enough to kill $\delta_0'$ but the deeper zero of $x^2$ at 0 is enough. Regarding, again condition \eqref{supports}, it was shown in \cite{lev2022example} that even under the assumptions $f > 0$ and $0 < \int f < \infty$ condition \eqref{supports} does not imply \eqref{soft}.

In almost all work that deals with translational tilings using the Fourier Transform condition \eqref{supports} is the starting point, which makes clear the fact that the shape of Fourier zeros of a function
\begin{equation}\label{fourier-zeros}
Z(f) = \Set{\xi \in\RR^d: \ft{f}(\xi) = 0}
\end{equation}
is of great importance in all problems concerning tilings by $f$ but in other geometric problems too, such as the Pompeiu problem \cite{kolountzakis2024curves,machado2023null}.

\subsection{An example from the Scottish book}\label{ss:scottish-book}

It is interesting to see by an example how pointing out a simple condition such as \eqref{supports} can make a big difference in our understanding of a problem. In the famous Scottish book \cite{mauldin1981scottish}, 
which was the notebook used at the Scottish caf\'e in Lviv (then Lwów) in the 1930s and 1940s by some of the well known heros of Mathematical Analysis (Stefan Banach, Mark Kac, Kazimierz Kuratowski, Hugo Steinhaus, Stanisław Ulam among them) to write down problems and solutions that they talked about while drinking coffee or liquor, one can find the following problem posed by H. Steinhaus.
\begin{quotation}
{\bf Problem 181}: Find a continuous function (or perhaps an analytic one) $f(x)$, positive and such that one has
$$
\sum_{n=-\infty}^\infty f(x+n) = 1
$$
(identically in $x$ in the interval $-\infty < x < +\infty$); examine whether $(1/\sqrt{\pi}) e^{-x^2}$ is such a function; or else prove the impossibility; or else prove uniqueness.
\end{quotation}
The condition asked to be satisfied by $f$ is of course the condition that it tiles the real line when translated by $\ZZ$. In light of \eqref{supports} and the very well known fact that the Fourier Transform of the gaussian $e^{-x^2}$ is a multiple of itself (and has thus no zeros at all) the power of notation, the right context and the era one lives in becomes clear, when one thinks that a giant such as Steinhaus posed this question without seeing the obvious negative answer for the gaussian (he did prove that the gaussian does not work in another way though \cite{mauldin1981scottish}).

\begin{figure}[h]
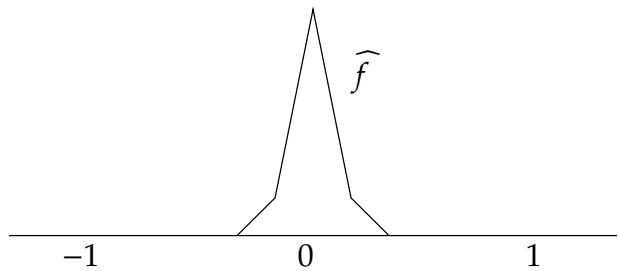
\caption{The Fourier Transform of the sum of two triangles functions of incommensurable bases is always positive.}\label{fig:triangles}
\end{figure}

We can easily solve this problem if we take as $f$ the Fourier Transform of a function which has compact support (this ensures analyticity for $f$), is positive definite (a function is positive definite for us if its Fourier Transform is non-negative) and we manage to ensure $f>0$. The Fourier Transform $\ft{f}$ must vanish on $\ZZ\setminus\Set{0}$, the support of $\ft{\delta_\ZZ}$ minus the point 0, according to the Poisson Summation Formula (see \S \ref{ss:psf}), but this is easy to achieve by taking $\ft{f}$ to be supported inside $(-1, 1)$. We use the well known fact that the triangle function
$$
S(x) = (1-\Abs{x})^+
$$
is a positive definite function. To see this observe that
$$
S(x) = \one_{[-\frac12, \frac12]}*\one_{[-\frac12, \frac12]}(x)
$$
and therefore
$$
\ft{S}(\xi) = \left(\frac{\sin(\pi \xi)}{\xi}\right)^2 \ge 0
$$
and the zeros of $\ft{S}$ are exactly at the integers apart from 0. We can then define $\ft{f}(\xi) = S(2\xi)+S(\sqrt2 \xi)$ as shown in Fig.\ \ref{fig:triangles} to obtain a function $\ft{f}$, supported in $(-1, 1)$, which is positive definite and with $f>0$ always, due to the fact that the only way a zero of $f$ could arise would be if both $\ft{S(2\cdot)}$ and $\ft{S(\sqrt2 \cdot)}$ vanished at the same time, which cannot happen because $\sqrt2$ is irrational.

\subsection{Lattices and the Poisson Summation Formula}\label{ss:psf}

\begin{figure}[h]
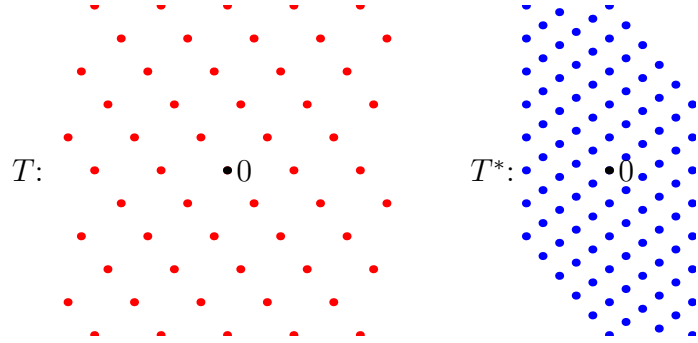

\begin{center}
\begin{asy}
size(9cm);

picture pic, picd;
real[][] a={{2, 1.2}, {0, 1}};
real[][] b=transpose(inverse(a));

int N=5;

for(int i=-N; i<=N; ++i)
 for(int j=-N; j<=N; ++j) {
  real[] x={i, j};
  real[] y=a*x;
  pair p=(y[0], y[1]);
  dot(pic, p, red);
  y = b*x;
  pair p=(y[0], y[1]);
  dot(picd, p, blue);
 }

dot(pic, "0", (0, 0), black);
dot(picd, "0", (0, 0), black);

clip(pic, (-N, -N)--(N, -N)--(N, N)--(-N, N)--cycle);
label(pic, "$T$:", (-1.2*N, 0));
clip(picd, (-N, -N)--(N, -N)--(N, N)--(-N, N)--cycle);
label(picd, "$T^*$:", (-0.7*N, 0));
add(pic);
add(shift((2.3*N, 0))*picd);
\end{asy}
\end{center}
\caption{A lattice $T$ and its dual lattice $T^*$.}\label{fig:dual-lattice}
\end{figure}

In the very important case when $T$ is a lattice $A\ZZ^d$ the Poisson Summation Formula in distributional form is
\begin{equation}\label{psf}
\ft{\delta_T} = \frac{1}{\vol T} \delta_{T^*}.
\end{equation}
That is, the Fourier Transform of unit point masses on a lattice $T$ are again point masses on the dual lattice $T^*$ of size $(\vol T)^{-1}$ (see Fig.\ \ref{fig:dual-lattice}. Here $\vol T$ is the volume of the fundamental domain of the lattice and $\vol T = \Abs{\det A}$ if $T = A \ZZ^d$. In this case, therefore, \eqref{supports} is a sufficient condition and we have the equivalence:
\begin{equation}\label{lattice-supports}
f*\delta_T = \mathrm{const.} \iff \ft{f} = 0 \text{ on } T^* \setminus\Set{0}.
\end{equation}
It should be clear of course that the same conclusion as \eqref{lattice-supports} can be drawn simply by the usual theory of multivariate Fourier series (followed perhaps by a linear transformation that will map $\ZZ^d$ to our lattice $T$). In the case $T = \ZZ^d$ the function
\begin{equation}\label{periodization}
F = f*\delta_T = \sum_{t \in \ZZ^d} f(x-t)
\end{equation}
is the so-called \textit{periodization} of $f$ and and is a $\ZZ^d$-periodic function, so it is constant if and only if all its non-constant Fourier coefficients $\ft{F}(n) = \ft{f}(n)$ are 0. This is precisely \eqref{lattice-supports}.

\subsection{Structure results about tilings via the Fourier Transform}\label{ss:structure}

This Fourier view of translational tiling has yielded many results both in the subject of tiling itself and in the Spectral Set problem.

One set of results (the first of which, due to D. Newman \cite{newman1977tesselations}, about tilings of the integers by finite sets, was proved using an easy combinatorial argument) is that one-dimensional tilings are essentially periodic and higher-dimensional tilings also have structure. In \cite{leptin1991uniform,lagarias1996tiling,kolountzakis1996structure} it was proved that whenever $0 \neq f \in L^1(\RR)$ is a function of compact support and $f*\delta_T$ is a tiling of the real line at some level, then $T$ is a finite union of complete arithmetic progressions
$$
T = \bigcup_{j=1}^J (a_j\ZZ+b_j),
$$
for some positive $a_j$ (this is a union as multisets, where multiplicities get added). If the tiling of $f$ by $T$ is \textit{indecomposable} (that is, it cannot be written as a superposition of tilings) then all the $a_j$ are commensurable (they are rational multiples of each other) and the tiling is periodic. The same conclusion holds if $f = \one_E$ is an indicator function of a set of positive and finite measure.

Where does the structure of $T$ (periodicity) come from? Newman's argument in \cite{newman1977tesselations} is essentially an application of the pigeon-hole principle and the fact that any tiling of the integers by a finite set is completely determined if we know it in a finite window (whose length depends on the tile).
In \cite{lagarias1996tiling} this argument was generalized to measurable sets in the line, a much more complicated case. Finally in \cite{kolountzakis1996structure} the problem was reduced to the so-called \textit{idempotent theorem} in Harmonic Analysis \cite{helson1953note,rudin1959idempotent,cohen1960conjecture} which greatly restricts the functions and measures whose Fourier Transform takes only the values 0 or 1 (or some finite set of values more generally).
(The paper \cite{leptin1991uniform}, also solving the same problem with the use of the idempotent theorem, predated the papers \cite{lagarias1996tiling,kolountzakis1996structure} but had gone unnoticed.) 

The idempotent theorem, valid in any locally compact abelian group \cite{cohen1960conjecture} is probably best understood in the case of the torus $\TT$ \cite{helson1953note}.
\begin{theorem}\label{thm:idempotent-torus}
If $\mu \in M(\TT)$ is a measure on $\TT$ whose Fourier coefficients $\ft{\mu}(n)$, $n\in\ZZ$, are either 0 or 1, then the set
$$
\Set{n\in\ZZ: \ft{\mu}(n) = 1}
$$
can be written as
$$
F \triangle \bigcup_{j=1}^J (a_j\ZZ+b_j)
$$
for a finite set $F \subseteq \ZZ$ and for some positive $a_j \in \ZZ$ (here $\triangle$ denotes symmetric difference of sets). In other words this set coincides with a periodic set up to finitely many terms.
\end{theorem}

We call such measures idempotent because all their convolution powers $\mu^{*n} = \mu * \cdots *\mu$ ($n$ times) are identical to $\mu$ (equivalently their Fourier coefficients take the values 0 or 1 only).

For an idempotent measure $\mu$ in a general locally compact abelian group $G$ the set of characters
$$
\Set{\gamma\in\ft{G}:\ \ft\mu(\gamma)=1}
$$
is in the so-called \textit{coset ring} of $\ft{G}$: this is the smallest algebra of subsets of $\ft{G}$ that contains all open cosets of $\ft{G}$, or, anything you can construct starting from some open cosets and performing finitely many set-theoretic operations.

To use the idempotent theorem in the study of tilings we use it in the form of the following theorem of Meyer \cite{meyer1970nombres}.
\begin{theorem}\label{thm:meyer}
Let $A\subseteq \RR$ be a set of bounded density and $\mu = \sum_{a \in A} c_a \delta_a$ be a measure, with $c_a \in S \subseteq \CC\setminus\Set{0}$, a finite set. If $\ft{\mu}$ is locally a measure and
$$
\Abs{\ft{\mu}([-R, R])} = O(R),\ \ \text{ as } R \to +\infty,
$$
then
\begin{equation}\label{coset-ring}
A = F \triangle \bigcup_{j=1}^J (\alpha_j \ZZ+ \beta_j),
\end{equation}
for some positive $\alpha_j$ and a finite set $F \subseteq \RR$.
(We do not view $A$ as a multiset here.)
\end{theorem}
It is worth noting that to prove Meyer's theorem (versions of which hold in any dimension) we use the idempotent theorem on $\ft{\RR_d}$, the group of continuous characters (dual group) of $\RR_d$, the real line $\RR$ equipped with the discrete topology. This is the so-called Bohr group or Bohr compactification of the real line, which we can view as an augmentation of $\RR$ which is such that the continuous functions on it are uniform limits of the usual trigonometric polynomials. This group is central in the theory of \textit{almost periodic functions} \cite{besicovitch1954almost,katznelson2004introduction}.

Assume now that $f$ is integrable and has compact support.
(See also \cite{kolountzakis2016non,kolountzakis2021tiling,lev2022example} for subtleties that may arise if the tile $f$ is not of compact support and also \cite{kolountzakis2019structure} for an analysis of the structure of \textit{multilpicative} tilings of the real line.)
Compact support of $f$ implies that its Fourier Transform $\ft{f}$ is analytic and therefore the set of Fourier zeros of $f$
$$
Z(f) = \Set{\ft{f} = 0}
$$
is a discrete set in $\RR$. Elementary complex analysis in the form of Jensen's formula implies that the number of zeros of $\ft{f}$ in $[-R, R]$ grows at most linearly in $R$. So we know that in a tiling of $f$ with $T$, from \eqref{supports}, $\ft{\delta_T}$ is a sum of distributions of point support at the points comprising $Z(f)$. By the growth of $T$ (bounded density) and a simple duality argument it follows that these distributions are actually point masses at $Z(f)$, uniformly bounded, so we can apply Theorem \ref{thm:meyer} to the measure $\mu = \delta_T$ and obtain the promised structure of $T$. The finite set $F$ in \ref{coset-ring} is easily seen to be empty (otherwise the support of $\ft{\delta_T}$ could not possibly be a discrete set).

\subsection{Structure of tilings in higher dimension}\label{ss:structure-plane-space}

The use of the idempotent theorem has led to structural results about translational tilings also in dimension 2 \cite{kolountzakis2000structure} and dimension 3 \cite{gravin2011translational,gravin2012structure} when the tile is a polygon or a polytope. In dimension $d\ge 2$ the situation becomes much more complicated than in dimension 1. The main reason is that the set of Fourier zeros $Z(f)$ of a function $f \in L^1(\RR^d)$ is no more a discrete set but rather a set of codimension 1. This allows for a much greater variety of tempered distributions to be supported on it (according  to \eqref{supports}) so new tricks are needed to derive structure of tilings. This structure is naturally much more flexible than the rigid situation we face in dimension 1.

\begin{figure}[h]
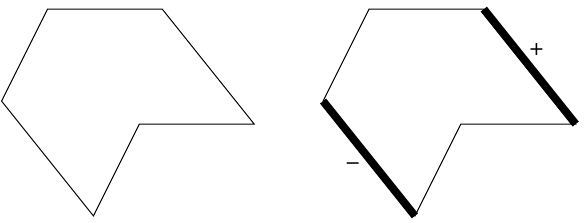
\caption{The indicator function of the polygon on the left has many Fourier zeros. If tiling occurs then also the edge measures that charge a pair of parallel edges with plus and minus their arc-length also tile, for every edge direction. Thus $\ft{\delta_T}$ lives on the intersection, as the direction of parallel edges changes, of all their Fourier zeros, typically a much smaller set.}\label{fig:polygon}
\end{figure}

Suppose for instance we are trying to decide if a polygon $E$ in the plane can tile by translations, at any (necessarily) integer level (see Fig.\ \ref{fig:polygon}). To use condition \eqref{supports} we should know about what tempered distributions are supported on the Fourier zeros $Z(\one_E)$. This is a union of some curves in the plane, and these support a great variety of tempered distributions, unlike the case of dimension 1, where the Fourier zeros were isolated points, and points can only support point masses and their derivatives.

The following observation saves us, at least in some nice cases in dimension 2. If $E$ tiles when translated by $T$, and $E$ contains only edges that come in parallel pairs (for instance, if $E$ is a symmetric convex polygon), then, it is intuitively clear, in the tiling of $E$ by $T$ each edge can only be ``countered'' in the tiling by its opposite edge (the one parallel to it and facing in the opposite direction) and these two edges must have the same length (else the shorter edges will not be able to counter the longer edge over a large tiled area). Fixing such an edge pair we can look at the measure $\mu$ in the plane which charges one edge by its arc-length measure and the opposite edge by negative its arc-length measure. See Fig.\ \ref{fig:polygon}. Then the translates of $\mu$ along $T$ completely kill each other. In other words $\mu$ tiles with $T$ at level 0, or $\mu*\delta_T = 0$, which, like in \eqref{supports}, leads us to the condition
\begin{equation}\label{edge-measures}
\supp{\ft{\delta_T}} \subseteq Z(\mu) = \Set{\ft{\mu} = 0}.
\end{equation}

\begin{figure}[h]
\ifdefined\SMART\resizebox{10cm}{!}{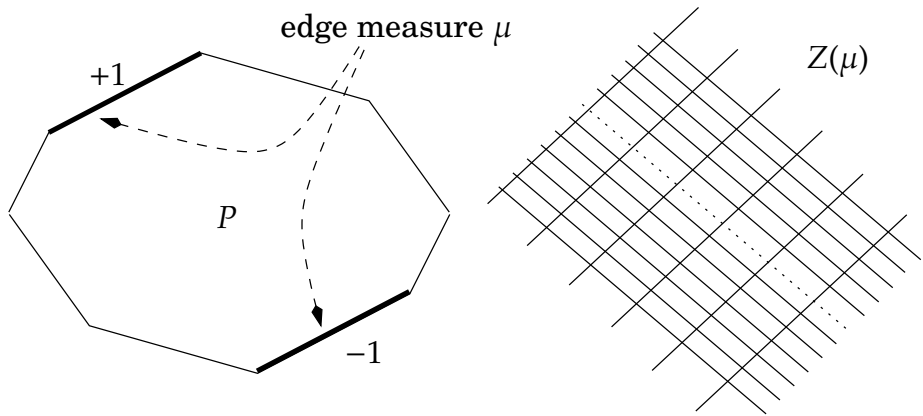}\else
\input 2dpoly-zeroset.pdf_t
\fi
\caption{The Fourier zeros of the edge measure $\mu$ defined on two opposite edges of the polygon $P$ is shown on the right. It consists of a union of two sets of parallel lines.}\label{fig:2dpoly-zeroset}
\end{figure}

But, it is easy to see, $Z(\mu)$ consists of a two collections of parallel lines, an easy set to work with. See Fig.\ \ref{fig:2dpoly-zeroset}. What is more important, \eqref{edge-measures} must hold for \textit{all parallel edge pairs} in the polygon $E$, so that $\ft{\delta_T}$ is supported on the intersection of all the Fourier zeros of edge measure pairs, and, generically, even inresecting two of them (two pairs of collections of parallel lines)) gives us a collection of isolated points, so that Meyer's theorem can now be applied essentially as in the case of tiling in dimension 1.

\subsubsection{Tilings of the line by translation and reflection}\label{sss:reflections}

What if one allows not just translations but reflections of $A$ as well in the tiling of $\ZZ$?

In \cite{lagarias1996tiling} it is pointed out that the set $A = \Set{0, 1, 5}$ can give two different tilings of the set $\Set{0, 1, \ldots, 8}$ if translations and reflections are allowed. Since we can tile $\ZZ$ by translates of $\Set{0, 1, \ldots, 8}$ and we can arbitrarily choose how (in which of the two available ways) we tile each translate with copies (translations and/or reflections) of $A$ it follows that we cannot expect such translation/reflection tilings to be periodic.

\begin{figure}[h]
\ifdefined\SMART\resizebox{10cm}{!}{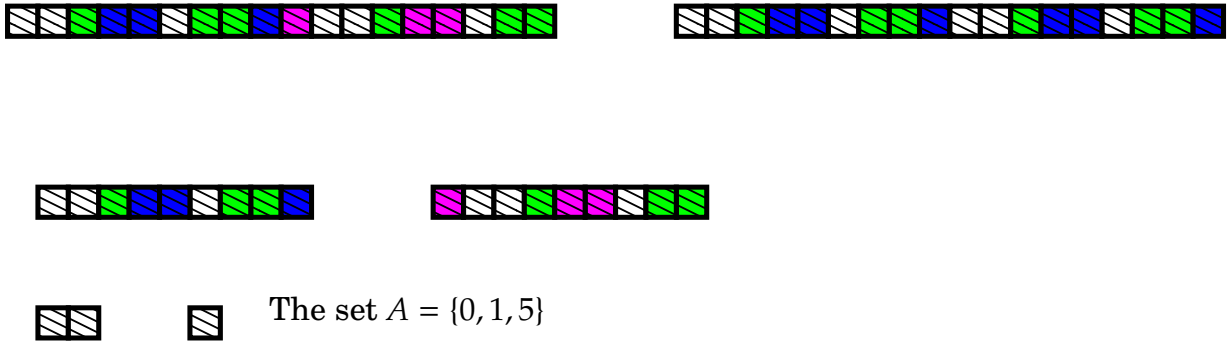}\else
\input reflections.pdf_t
\fi
\caption{
The set $A = \Set{0, 1, 5} \subseteq \ZZ$ shown at the bottom, can tile using translations \textit{and reflections} the interval $\Set{0, 1, \ldots, 8}$ in two different ways, shown in the middle row above. Each copy used is drawn in a different color (in fact these two ways are reflections of each other).
In the top row we show two fragments of tilings of the line by using one of the two tiled intervals of the middle row at will. In the top left fragment we have used the first tiled interval and then the second, and in the top right fragment we have used the first tiled interval twice. Notice that these two fragments are indeed different tilings of $\ZZ$ as they show different patterns. In the top right fragment we only observe two points of the same tile followed by one point of another tile, while in the top left we can also see a single point of one tile followed by a single point of another tile.
}\label{fig:reflections}
\end{figure}

But lack of periodicity does not necessarily mean lack of structure.

\qq{
What kind of structure may we expect from a translation/reflection tiling of $\ZZ$ by a finite set $A$?

In particular is it true that non-periodicity implies that there is a translational tile $A'$ of $\ZZ$ that can be tiled in two or more different ways by copies (translations and reflections) of $A$?

The set $A'$ here is the analogue of the set $\Set{0, 1, \ldots, 8}$.
}

There are two ways to view this question:
\begin{enumerate}
\item \textbf{Every tiling has structure.}\\
Fix a tiling $\mathcal{T}$ by $A$ which uses translations and reflections. Is it true that
\begin{quote}
{\bf Property S}:
there is a translational tile $A'\subseteq\ZZ$ which tiles $\ZZ$ by a set of translates $B'$, such that every translate
$$
A'+b', \text{ for some } b' \in B',
$$
is tiled (with translations and reflections) by $A$ in the tiling $\mathcal{T}$?
\end{quote}
\item \textbf{Every tile has a structured tiling.}\\
If $A$ can tile $\ZZ$ using translations and reflections then is there such a tiling $\mathcal{T}$ using $A$ which has Property S (as above)?
\end{enumerate}

\subsection{Spectrality as a tiling question}\label{ss:spectrality-as-tiling}

Having discussed translational tiling somewhat at length let us now return to the Spectral Set question: why make the conjecture that the spectral sets are precisely the sets that can tile? What could possible connect the two concepts? 

Suppose $f \in L^2(E)$ and suppose also that the set of exponentials
$$
E(\Lambda) = \Set{e_\lambda(x):\ \lambda \in \Lambda},
$$
for some $\Lambda \in \RR^d$, is orthogonal in $E$. Bessel's inequality in the Hilbert space $L^2(E)$ then becomes
$$
\sum_{\lambda \in \Lambda} \Abs{\inner{f}{e_\lambda}}^2 \le \vol(E) \Norm{f}_2^2 
$$
If $E(\Lambda)$ is complete in $L^2(E)$ as well as orthogonal then Bessel's inequality becomes an equality, and this is equivalent to completeness.

Assuming completeness then apply this now to $f(x) = e_t(x)$ for an arbitrary frequency $t \in \RR^d$ to get
\begin{equation}\label{spectrum-as-tiling}
\sum_{\lambda\in\Lambda} \Abs{\ft{\one_E}}^2 (t-\lambda) = \vol(E)^2.
\end{equation}
But this equation precisely means that the integrable, nonnegative function
$\Abs{\ft{\one_E}}^2$ (the \textit{power spectrum} of $\one_E$)
tiles $\RR^d$ when translated at the locations $\Lambda$ at level $\vol(E)^2$. And it is not hard to see, due essentially to the density of trigonometric polynomials in $L^2(E)$, that \eqref{spectrum-as-tiling} implies $\sum_{\lambda \in \Lambda} \Abs{\inner{f}{e_\lambda}}^2 = \vol(E) \Norm{f}_2^2$ for any $f \in L^2(E)$ (completeness).

\begin{figure}[h]
\ifdefined\SMART\resizebox{10cm}{!}{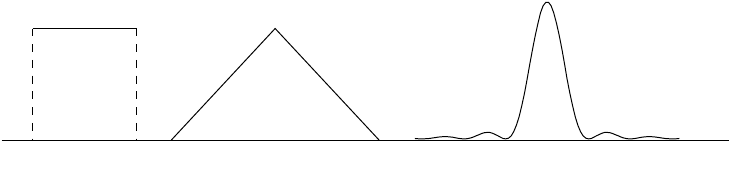}\else
\input fuglede-interval.pdf_t
\fi
\caption{The functions $\one_E$, $\one_E*\one_{-E}$, $\Abs{\ft{\one_E}}^2$ when $E$ is an interval.}\label{fig:fuglede-interval}
\end{figure}

We may thus restate the Fuglede conjecture as follows:
\begin{quotation}
\textbf{Fuglede or Spectral Set conjecture}: $E$ tiles $\iff$ $\Abs{\ft{\one_E}}^2$ tiles.
\end{quotation}
In this new form we hope that the conjecture makes a little more sense.

At this point, having interpreted spectrality as the tiling \eqref{spectrum-as-tiling}, we can easily derive some properties that the spectrum $\Lambda$ must have if it exists. The first property is that it must have a density and this density is equal to $\vol(E)$. More precisely, we have
that for any $\epsilon>0$ there exists a $R_0$ such that whenever $R>R_0$ the number $n$ of points of $\Lambda$ in any ball $B$ of radius $R$ satisfies
$$
(\vol(E)-\epsilon)\, \vol(B) \le n \le (\vol(E)+\epsilon)\, \vol(B).
$$
This is nice: the density of the spectrum of $E$ is equal to the volume of $E$.

Another property of $\Lambda$, easily derivable from orthogonality alone, is that the distance $\Abs{\lambda-\mu}$, $\lambda, \mu \in \Lambda$, is bounded below by a positive number. Indeed, we know that orthogonality of $e_\lambda$ and $e_\mu$ over $E$ is equivalent to the vanishing
$$
\ft{\one_E}(\lambda-\mu) = 0.
$$
But $\ft{\one_E}(0) = \vol(E) > 0$ and $\ft{\one_E}$ is a continuous function, so there is $r>0$ such that $\Abs{\ft{\one_E}(x)}^2$ is positive in $\Set{\Abs{x}\le r}$,
so $\Abs{\lambda-\mu} > r$.

We often say that the spectrum is a \textit{well distributed set}: a set whose elements are at least a positive distance apart and does not have arbitrarily large gaps: there is a positive $R$ such that every ball of radius $R$ contains at at least one point of the spectrum.

\subsection{Proof of the lattice Fuglede Conjecture}\label{ss:lattice-fuglede}

Let us now prove Theorem \ref{thm:lattice-fuglede}: the set $E$ lattice-tiles $\RR^d$ if and only if it has a lattice spectrum.

Let us assume for simplicity that $\vol(E) = 1$ so that spectrality is equivalent to the tiling $\Abs{\ft{\one_E}}^2 * \delta_{T^*} = 1$, where $T^*$ is a lattice. One immediate consequence of this is that the density of $T^*$ (and, hence, also that of $T$) is 1. This is intuitively obvious (in any tiling the product of the integral of the tile and the density of the translates must equal the level of the tiling -- compare the volume of a large tiled ball with the sum of the volumes of the tiles contributing to tiling that ball) but see also \cite{kolountzakis1996structure}.

Since, by the Poisson Summation Formula \eqref{psf}, $\ft{\delta_{T^*}}$ is a measure we have that spectrality is equivalent, according to \eqref{supports}, to $T \subseteq \Set{0} \cup Z(\Abs{\ft{\one_E}}^2)$. But the Fourier Transform of $\Abs{\ft{\one_E}}^2$ is $\one_E * \one_{-E}$, whose zero set is $(E-E)^c$ (it is easier to think of the case of an open set $E$ here, as this removes some null-set considerations from the argument). So we conclude that spectrality is equivalent to 
\begin{equation}\label{T-packing}
T \cap (E-E) = \Set{0}.
\end{equation}
Since $T$ is a lattice we have $T = T-T$ and condition \eqref{T-packing} is easily seen to be equivalent to packing: the translates $E+t$, $t \in T$, are disjoint. If these translates failed to cover everything and left some set of positive measure uncovered, this ``hole'' would repeat itself indefinitely for this is a periodic arrangement since $T$ is a lattice. This in turn would imply that $\vol(E) \cdot \dens(T) < 1$, which clearly cannot happen.

\subsection{Filling a box with bricks of two types}\label{ss:bricks}

Finally, let us point to an amusing application of the Fourier method for tilings \cite{bower2004box,kolountzakis2004box}. Suppose we are given a rectangular box in $\RR^d$ and we are to fill it exactly (a tiling) by using only two kinds of bricks of dimensions, say, $a_1 \times \cdots \times a_d$ and $b_1 \times \cdots \times b_d$. We have an infinite supply of both kinds of bricks but we are only allowed to translate them, not turn them in any way. In other words the bricks we are using are always parallel to each other and parallel to the box that is to be filled.

\begin{figure}[h]
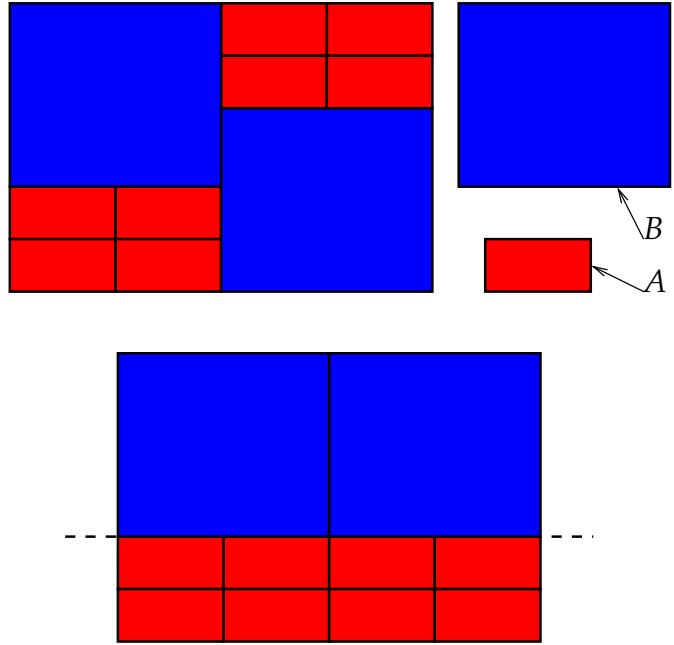
\caption{A box $A$ can filled exactly with two type of bricks, $A$ and $B$, if and only if it can be cut in two parts each of which can be filled with bricks of one type only.}\label{fig:box-filled}
\end{figure}

Using Fourier Analysis one can prove the following, first shown without Fourier Analysis in \cite{bower2004box} (see Fig.\ \ref{fig:box-filled}).
\begin{theorem}\label{thm:bricks}
A rectangular $d$-dimensional box can be filled exactly with bricks of two different kinds $A$ and $B$ if and only if the box can be cut into two rectangular boxed along one of its sides so that brick type $A$ can fill exactly one part and brick type $B$ can fill exactly the other part.
\end{theorem}
As shown in Fig.\ \ref{fig:three} Theorem \ref{thm:bricks} fails if we allow for three types of bricks, even in dimension 2.

\begin{figure}[h]
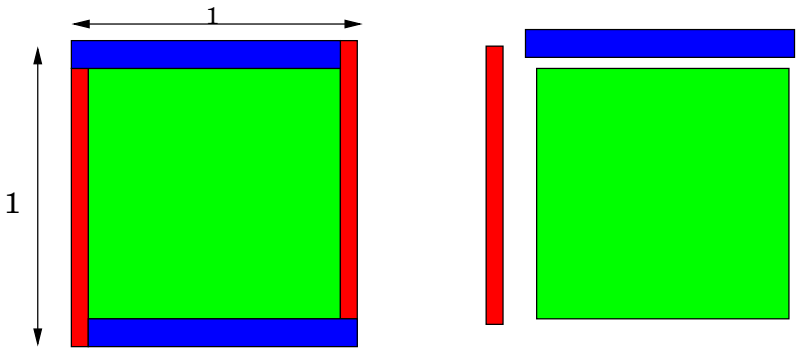
\caption{Theorem \ref{thm:bricks} fails if we are using three types of bricks in order to fill a box, even in dimension 2. The $1\times 1$ box on the left can be filled with the three types of bricks shown, but cannot be decomposed into smaller boxes each of which can be filled with fewer types of bricks (by inspection).}\label{fig:three}
\end{figure}

\begin{figure}[h]
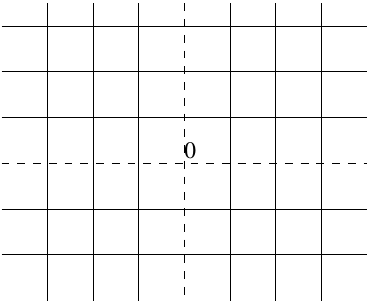
\caption{The Fourier zeros of the unit square in the plane is two collections of parallel lines, parallel to the axes, of spacing 1. The axes themselves are missing except for the integer points.}\label{fig:zeros}
\end{figure}

The Fourier Transform of the indicator function of the box
$$
C = \left(-\frac{c_1}{2}, \frac{c_1}{2}\right) \times \cdots \times \left(-\frac{c_d}{2}, \frac{c_d}{2}\right)
$$
is easily seen to be the function
$$
\ft{\one_C}(\xi) = \frac{\sin(\pi c_1 \xi_1)}{\xi_1} \cdots \frac{\sin(\pi c_d \xi_d)}{\xi_d}
$$
so that the Fourier zeros of $\one_C$ are exactly the $\xi \in \RR^d$ such that for some $j=1, 2, \ldots, d$ the coordinate $\xi_j$ is a \textit{non-zero} multiple of $\frac{1}{c_j}$. This simple characterization is central to many results about rectangle tilings or rectangle spectra \cite{kolountzakis1998notched,lagarias2000orthonormal,iosevich1998spectral,kolountzakis2000packing}. See \cite{kolountzakis2004box} for the details of the proof of Theorem \ref{thm:bricks}.

\section{Fuglede Conjecture prehistory}


In this section we describe several of the results that were given before the Fuglede conjecture was disproved in its generality in dimension at least 3 \cite{tao2004fuglede,kolountzakis2006tiles,kolountzakis2006hadamard,farkas2006onfuglede,farkas2006tiles}. The exceptions are \S \ref{ss:periodicity-of-spectrum} and \S \ref{ss:product} which refer to results obtained after the disproof in 2004. All the results that we are describing here were in the direction of supporting the conjecture. Most of them were of the form: if sets that tile have a property $P$ then so do spectral sets, or vice versa.

\subsection{Non-symmetric convex bodies are not spectral}\label{ss:nonsym}

Our first result has to do with convex sets. It was known since the time of Minkowski that for a convex body to tile $\RR^d$ it is necessary that it is symmetric about a point. This point can be taken to be $0 \in \RR^d$ so in that case symmetric $K$ means $K = -K$.

Let us see the proof. Suppose $K$ is a convex body in $\RR^d$ and suppose also that $K$ tiles when translated at the locations $T\subseteq\RR^d$. By just packing (non-overlapping copies) we have $(K-K) \cap (T-T) = \Set{0}$. Since $K-K$ is again a convex body, symmetric this time about 0, it follows that if $L=\frac12 (K-K)$ we have $L-L = K-K$ hence $L$ also packs when translated at $T$.

By the Brunn-Minkowski inequality \cite{gardner2002brunn} we have for the convex body $K$ that
$$
\vol(\frac12(K-K)) \ge \vol(K)
$$
with equality if and only if $K$ is symmetric.

Thus if $K$ is non-symmetric it follows that $\vol(L) > \vol(K)$. But this makes the packing of $L$ by translates at $T$ impossible: since $K$ tiles with $T$ we have that $\vol(K)\cdot\dens(T)=1$ and in the packing by $L$ we must have $\vol(L)\cdot\dens(T)\le 1$, a contradiction, so $K$ must be symmetric.

This is also true for spectral sets \cite{kolountzakis2000nonsymmetric}:
\begin{theorem}\label{thm:nonsym}
If $E$ is bounded, open, convex and nonsymmetric then it is not spectral. 
\end{theorem}

\begin{proof}

Suppose that $K$ has measure 1 (this is no restriction of course) is convex, non-symmetric and spectral, with spectrum $\Lambda \subseteq \RR^d$. Then, from \eqref{spectrum-as-tiling}, we have the tiling
$$
\sum_{\lambda \in \Lambda} \Abs{\ft{\one_K}}^2(x-\lambda) = 1,
$$
for almost all $x \in \RR^d$. By our condition \eqref{supports} and writing $f=\Abs{\ft{\one_K}}^2$ we then have
$$
\supp \ft{\delta_\Lambda} \subseteq \Set{\ft{f} = 0} \cup \Set{0}.
$$
But $\ft{f} = \one_K*\one_{-K}$ so that $\Set{\ft{f}=0} = (K-K)^c$.

By the non-symmetry of $K$ and writing $H=\frac12(K-K)$ we have $\vol(H) > \vol(K) = 1$ by the equality case in the Brunn-Minkowski inequality once again. But the support of the function $\one_H*\one_H$ is $H+H=K-K$. Take a number $\alpha<1$ very close to 1 and define $\ft{g} = \one_{\alpha H} * \one_{\alpha H}$ which has support $\alpha(K-K)$ a set which a positive distance from $(K-K)^c$, the support of $\ft{\delta_\Lambda}$. We claim that the function $g$ tiles with the translation set $\Lambda$. The reason is that \eqref{supports} holds in this case with room to spare. Not only is $\supp\ft{\delta_\Lambda}$ contained in $\Set{\ft{g}=0} \cup \Set{0}$ but it is contained in its interior, so the tempered distribution $\ft{\delta_\Lambda}$ is killed by $\ft{g}$, except at the origin, and $g$ tiles with $\Lambda$ at level
$$
\int g \cdot \dens\Lambda = \ft{g}(0) \cdot 1 = \one_{\alpha H} * \one_{\alpha H} (0) = \vol(\alpha H).
$$
But the value $g(0) = \Abs{\ft{\one_{\alpha H}}}^2(0) = \vol(\alpha H)^2$ can be made to be $ > \vol(\alpha H)$ (remember $\vol(H)>1$) for some $\alpha<1$. This is a contradiction as the value of the tile $g$ at some point cannot be higher than the level of the tiling.

\end{proof}

\subsection{The Tur\'an extremal problem about positive definite functions of given support}\label{ss:turan}

In the proof of Theorem \ref{thm:nonsym} the following was the key: we had a symmetric convex set, namely $K-K$ ($K$ itself was not symmetric) and the following two functions were supported in it:
$$
\alpha = \one_K*\one_{-K} \text{ and } \beta = \one_H*\one_H
$$
where $H = \frac12(K-K)$ (thus $H$ is symmetric). It was essential that both these functions were positive definite (i.e. had a nonnegative Fourier Transform) and what made the proof work was the fact that
$$
\int \beta > \int \alpha,
$$
by the Brunn-Minkowski inequality (or, rather, the equality case in that inequality).

This means that if we have a (symmetric) convex body $P$ and we manage to find a positive definite function $\gamma$ supported inside $P$ whose integral $\int \gamma$ is strictly larger than the integral of the positive definite function
$$
\one_{\frac12 P}*\one_{\frac12 P},
$$
which is also supported inside $P$, then we immediately have that $P$ cannot be spectral.

This maximization problem
\begin{quotation}
Given $\Omega \subseteq \RR^d$ (containing $0$ and symmetric with respect to $0$) maximize $\int \gamma$ where $\gamma$ is a continuous positive definite function supported in $\Omega$ with $\gamma(0)=1$.
\end{quotation}
is the so-called Tur\' an extremal problem about positive definite functions of given support \cite{stechkin1972extremal}. The conjecture for convex $\Omega$ is that the extremal function is indeed the autoconvolution of the half-body $\one_{\frac12\Omega}*\one_{\frac12\Omega}$.

Our discussion implies that the Tur\' an conjecture is true for all spectral convex bodies \cite{kolountzakis2003problem,arestov2001hexagon,arestov2002turan}. By the results described later in \S \ref{ss:using-w-tiling}, the class of convex spectral bodies is precisely the class of convex translational tiles.

The Tur\'an extremal problem is still wide-open \cite{kolountzakis2006turan}, even for the case of convex bodies. Besides convex tiles it is also known for the Euclidean ball (quite far from being a tile) \cite{gorbachev2001extremum,kolountzakis2003problem}, but it is not known for other symmetric convex sets. See also \cite{gorbachev2004turan,andreev1996extremum}.

One case is easy to see. If every positive definite function, among those competing for the maximal integral, can be written as a sum of convolution squares then we can get the right bound. A function $f$ is a sum of convolution squares if
$$
f = g_1*\widetilde{g_1}+g_2*\widetilde{g_2}+\cdots,
$$
where $\widetilde{g}(x) = \overline{g(-x)}$. If this happens then we have
$$
\ft{f} = \Abs{\ft{g_1}}^2+\Abs{\ft{g_2}}^2 + \cdots
$$
so $f$ is positive definite. If the $g_i$ are supported in the half body $\frac12 K$ then the $g_i*\widetilde{g_i}$ are supported in $K$.
Then
\begin{align*}
\int f &= \sum_j \int g_j*\widetilde{g_j}\\
  &= \sum_j \Abs{\int g_j}^2\\
  &\le \sum_j \int\Abs{g}^2 \cdot \frac{1}{2^d} \vol(K) \text{ {\hskip 3em} (Cauchy-Schwarz) }\\
  &= f(0) \frac{1}{2^d}\vol(K),
\end{align*}
which is precisely the conjectured inequality.
This does happen in dimension 1: every positive definite function on the interval $[-A, A]$ is the convolution square of a function supported on $[-A/2, A/2]$, so the Tur\'an Conjecture holds for an interval and the extremal function is the triangle function
$$
(1-\Abs{x}/A)^+
$$
(whose Fourier Transform is the nonnegative, and very important, Fej\'er kernel).

A similar decomposition holds also for the case of \textit{radial}, positive definite smooth functions supported on a $0$-centered ball in any dimension: any such function can be expanded as $\sum_{n=1}^\infty g_j * \widetilde{g_j}$, where the functions $g_j$ are supported in the half-ball \cite{rudin1970extension}. This was used in proving \cite{kolountzakis2003problem} that the Tur\'an conjecture holds for the ball.

\qq{Prove the Tur\' an Conjecture for the regular octagon, the simplest symmetric convex set in the plane that does not tile.}

Another interesting question that arose from the study of positive definite functions with restricted support is that of approximation by other positive definite functions of even smaller support \cite{kolountzakis2006turan}.

Suppose that $\Omega$ is a bounded open set with a nice boundary and that $f$ is a smooth, say, function supported in $\overline{\Omega}$ (the support of a function is a closed set, so this means that the function may be non-zero all the way up to the boundary of $\Omega$, where it has to vanish for continuity). It is not hard to see that we can approximate $f$, say, uniformly in $\Omega$, by another smooth function $g$ whose support is strictly inside $\Omega$. In other words $g$ is zero in a neighborhood of $\partial\Omega$ (chop off $f$ a little inside $\Omega$ and convolve the resulting function with an appropriately narrow smooth approximate identity to construct such a function $g$).

What if $f$ is assumed to be positive definite and $g$ is required to also be positive definite? The previous chopping off by which we constructed $g$ destroys the positive definiteness so this method does not work. If $\Omega$ is convex, or even strictly star-shaped (this means that for $0<t<1$ we have $\overline{t\Omega} \subseteq \Omega$) we can replace this brutal chopping off by taking the function $f(x/t)$ which is supported strictly inside $\Omega$, is still positive definite and is uniformly close to $f$ by the uniform continuity of $f$. But if the domain is not strictly star shaped then we do not know if this is possible.

\qq{\label{q:approx} If $0 \in \Omega \subseteq \RR^d$ is an open set, symmetric about $0$, and with a piecewise smooth boundary and $f$ is a positive definite function which is non-zero only in $\Omega$ and $\epsilon>0$ can we always find a positive definite function $g$ which is non-zero only in the set $\Set{x\in\Omega:\ \mathrm{dist}\,(x, \partial \Omega)>\delta}$ for some positive $\delta$ and is such that $\Abs{f(x)-g(x)}<\epsilon$ for $x \in \RR^d$.
}

Few partial results exist for this problem. In \cite{mavroudis2013approximation} it is shown that the answer to Question \ref{q:approx} is indeed affirmative in dimension 1 when $\Omega$ is a finite union of intervals symmetric about 0. This is really due to the fact that in dimension 1 the boundary of such a set, a finite collection of points, is small enough compared to the set that the chopping off followed by smoothing approach can indeed work. This method cannot be used in dimension 2 and higher. In the case when $\Omega\subseteq \RR^d$ has radial symmetry and is a union of a ball centered at 0 and finitely many annuli also centered at 0, the function $f$ is also radial then the answer is also shown to be affirmative in \cite{mavroudis2013approximation}.

\subsection{Unbalanced polytopes are not spectral}\label{ss:unbalanced}

There is a different proof of the necessity of symmetry for spectral convex domains which are \textit{polytopes} \cite{kolountzakis2002class}.

It is intuitively clear that if we have a polytope $P$ (not necessarily convex) which tiles by translation, then, if we look at its co-dimension 1 faces which are normal to a given direction $u$ then the total area of such faces whose exterior normal vector points in the positive $u$ direction (call them $u^+$ faces) must be the same as the total area of those faces whose exterior normal vector points in the negative $u$ direction (the $u^-$ faces). The reason is that in any tiling by translations of $P$ the $u^+$ faces can only be countered from the outside by $u^-$ faces, so their total areas must match. See Fig.\ \ref{fig:faces}.

\begin{figure}[h]
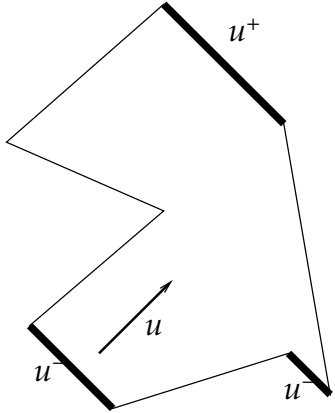
\caption{In a tiling polytope the total area of the co-dimension 1 faces that point in the positive $u$ direction ($u^+$ faces) must be the same as that of those faces pointing in the negative $u$ direction ($u^-$ faces). The same is true for spectral polytopes \cite{kolountzakis2002class}.}\label{fig:faces}
\end{figure}

It is proved in \cite{kolountzakis2002class} that if a polytope is unbalanced (different total areas for the $u^+$ and $u^-$ faces, for some direction $u$) then it cannot be spectral. To see why this implies that non-symmetric convex polytopes cannot be spectral we need to invoke a classical theorem of Minkowski, which implies that a non-symmetric polytope has a non-symmetric surface area measure. This is the measure on the sphere $S^{d-1}$ which is defined by pushing the surface measure from the boundary of the polytope onto the sphere via the Gauss map, which maps (almost) every point on the polytope to its exterior unit normal vector \cite{schneider2013convex}. To have a non-symmetric surface area measure, for a polytope, means precisely that it is unbalanced, and Minkowski's theorem says that a non-symmetric polytope is unbalanced, thus non-spectral.

Let us indicate how we prove (for the details see \cite{kolountzakis2002class}) that unbalanced polytopes are not spectral. Say that $u \in \RR^d$ is a bad direction for the unbalanced polytope $E$, that is the surface area of $E$ pointing in the positive $u$ direction is different from that pointing in the $-u$ direction. Differentiate the indicator function $\one_E$ of the polytope along the direction $u$, in the sense of distributions. The interior of $E$ vanishes of course as the function is constant there. All that remains is a measure supported on the boundary $\partial E$. This measure is constant on each co-dimension 1 face since all that matters is the angle of $u$ and that face. It is 0 only on the faces parallel to $u$. On the faces normal to $u$ it is equal to $1$ on those facing opposite $u$ and equal to $-1$ on those facing in the direction of $u$. Call $\nu$ this measure on $\partial E$. Since $\nu = \partial_u \one_E$ we have
\begin{equation}\label{ft-nu}
\ft{\nu}(\xi) = (2\pi i) (\xi \cdot u) \ft{\one_E}(\xi),\ \ \ (\xi \in \RR^d). 
\end{equation}

\begin{figure}[h]
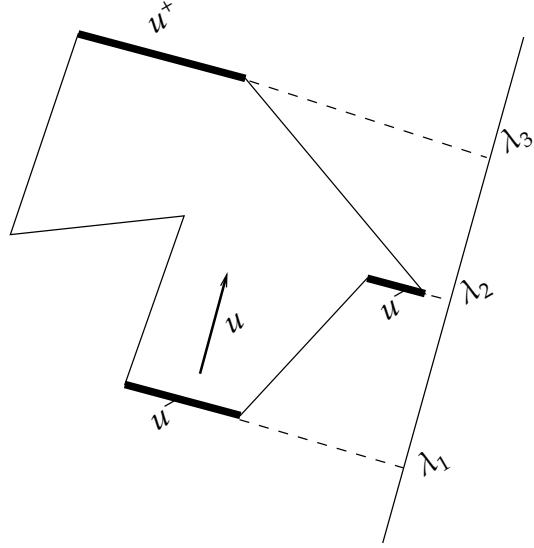
\caption{We project the surface measure of our body onto a one-dimensional subspace $\RR u$. We obtain a measure on the line with a continuous part and an atomic part supported on finitely many points $\lambda_j$.}\label{fig:projection}
\end{figure}

Project the measure $\nu$ onto the one-dimensional subspace $\RR u$ to obtain a measure $\mu$ on the line (see Fig.\ \ref{fig:projection}). The parts of $\nu$ normal to $u$ will give rise to an atomic part $\mu_a$ of $\mu$ of the form $\sum_{k=1}^K c_k \delta_{\lambda_k}$ for some $\lambda_k \in \RR$ and some coefficients $c_k \in \RR$ which are such that $\sum_{k=1}^K c_k \neq 0$ (this is a consequence of the polytope being unbalanced in the direction $u$). The part of $\nu$ that is not normal to $u$ contributes a continuous part to $\mu$, call it $\mu_c$, so that $\mu = \mu_a + \mu_c$.

It is a general and well known fact that if we project a function or measure defined on $\RR^d$ orthogonally onto a subspace $V$ of $\RR^d$ and then take the Fourier Transform of the projection on $V$ we will read precisely the restriction on $V$ of the Fourier Transform of the original function or measure on $\RR^d$. This is a simple consequence of Fubini's theorem. Therefore, using \eqref{ft-nu}, for all $t \in \RR$ we have
$$
\ft{\mu}(t) = \ft{\nu}(t\, u) = (2\pi i) \, t \, \ft{\one_E}(t\, u).
$$
By the Riemann Lebesgue theorem applied to $\mu_c$ we have that, as $t \to \infty$,
$$
\ft{\mu}(t) = \ft{\mu_a}(t) + o(1) = \sum_{k=1}^K c_k e^{2\pi i \lambda_k t}
$$
Let us now pretend that $\ft{\mu_a}(t)$ is periodic in $t$ with period, say, $T$. Of course this is not in general true, but it is \textit{almost periodic} \cite{besicovitch1954almost} and this turns out to be sufficient for the argument we are about to describe. Since $\ft{\mu_a}(0) = \sum_{k=1}^K c_k \neq 0$ it follows that $\ft{\mu_a}$ is absolutely larger than a positive constant in some interval $(-\delta, \delta)$ and, by periodicity, at all intervals $(-\delta, \delta)+T\ZZ$.

Since $\ft{\mu_c}$ tends to 0 it follows that $\ft{\mu}$ is also non-zero on that arithmetic progression of intervals, at least after some point $nT$. On the line $\RR u$ the functions $\ft{\mu}$ and $\ft{\one_E}$ have the same zeros (except at 0), so $\ft{\one_E}$ also does not vanish on $(-\delta, \delta)\pm T \ZZ^{\ge n} u$. By the uniform continuity of $\ft{\one_E}$ on $\RR^d$ it follows that, possibly for a smaller positive $\delta$, $\ft{\one_E}$ does not vanish on the union of balls (see Fig.\ \ref{fig:ap-balls})
\begin{equation}\label{non-zeros}
U = B_\delta(0) \pm T \ZZ^{\ge n} u.
\end{equation}
Let us see now how this contradicts the fact that a spectrum $\Lambda$ of $E$ must have positive density. Assume for simplicity that $u = e_d$ is the $d$-th coordinate unit vector and view $\RR^d$ as covered by a finite union of translates of the lattice arrangement of balls
$$
(T\ZZ)^d + B_\delta.
$$
At least one of them must contain a part of $\Lambda$ that is of positive density. However the presence of one point of $\Lambda$ in one such ball implies that all balls above it (in the positive $d$-th axis direction) above height $nT$ are empty of $\Lambda$-points: any $\Lambda$-point in there would contradict the fact that there are no zeros of $\ft{\one_E}$ in $U$ of \eqref{non-zeros}. So a half-space of balls is empty, contradicting that $\Lambda$ has positive density in that lattice arrangement of balls.

\begin{figure}[h]
\ifdefined\SMART\resizebox{10cm}{!}{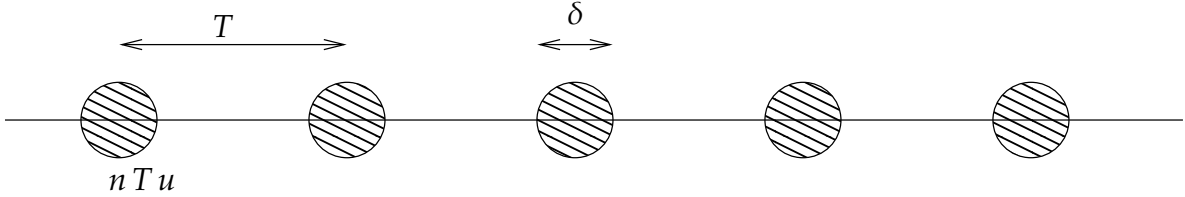}\else
\input ap-balls.pdf_t
\fi
\caption{The function $\ft{\one_E}$ is non-zero on this arithmetic progression of $\delta$-balls spaced by $T$, after some initial point $nT u$.}\label{fig:ap-balls}
\end{figure}

The fact that every spectral polytope must have balanced area measure, as we showed, has been vastly extended in \cite{lev2021spectrality}, to include all facets of all co-dimensions, not just of co-dimension 1. More precisely, they show that all \textit{Hadwiger functionals} of the polytope are 0. As a consequence, it is proved in \cite{lev2021spectrality} that spectral polytopes are equidecomposable by translations to a cube, i.e., we can cut them up into finitely many polytopes which can then be rearranged by translations to form a cube.

\subsection{Convex tiles are spectral}\label{ss:convex-tiles-are-spectral}

When we restrict our attention to convex sets, regarding the validity of the Fuglede Conjecture, we, at last, have some good news: the convex tiles are all spectral, so that one direction at least of the Fuglede Conjecture holds. The reason feels like a let-down though: there are no complicated convex tiles, we know them all, to some extent, and they are all lattice tiles as well. This, by Theorem \ref{thm:lattice-fuglede}, implies that they are also spectral, having as one of their spectra the dual lattice of the lattice they tile with.

More precisely, it has long been known \cite{venkov1954class,mcmullen1980convex} that a convex body $K$ tiles by translations if and only if all the following conditions hold:
\begin{itemize}
\item 
$K$ is a polytope
\item 
$K$ is symmetric
\item 
$K$ has symmetric co-dimension 1 faces
\item 
Every \textit{belt} of $K$ consists of four or six faces. For the precise definition of a belt see, for instance, \cite{lev2022fuglede}.
\end{itemize}
For instance, in dimension 2, where the codimension 2 faces are just vertices of the polytope these conditions say that a convex polygon tiles by translations if and only if it is a parallelogram or a symmetric hexagon. See Fig.\ \ref{fig:planar-convex}.

It is also known that whenever a convex body can tile by translations then it can also tile by lattice translations.

\begin{figure}[h]
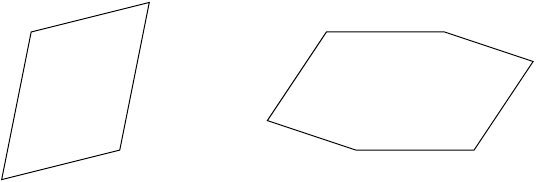
\caption{The only planar convex sets that tile by translation or are spectral are the parallelograms and the symmetric hexagons.}\label{fig:planar-convex}
\end{figure}

Early on the Fuglede Conjecture (both directions) was settled for planar convex bodies \cite{iosevich2003fuglede}:
\begin{theorem}
A planar convex set is a translational tile or a spectral set precisely when it is a parallelogram or a symmetric hexagon (see Fig.\ \ref{fig:planar-convex}).
\end{theorem}

We will see in \S \ref{ss:using-w-tiling} that the other direction of the Fuglede Conjecture is also true for convex bodies. This is one of the greatest developments in this area in the last decades.

\subsection{The spectra of the cube}\label{ss:spectra-of-cube}

In \cite{lagarias2000orthonormal,iosevich1998spectral,kolountzakis2000packing} the following rather idiosyncratic result was proved for the cube in $\RR^d$.
\begin{theorem}\label{thm:spectra-of-cube}
Let $Q = \left(-\frac12, \frac12\right)^d$ be the unit cube in $\RR^d$ and $T \subseteq \RR^d$. Then
$$
T \text{ is a spectrum of } Q \iff Q \text{ tiles when translated by } T.
$$
\end{theorem}
It should be said here that the sets $L$ the cube tiles with (the ``tiling complements'' of the cube) can be quite exotic. It used to be a conjecture of Keller \cite{keller-conjecture} from the 1930s that in any translational tiling by a cube one can find two cubes which share a whole co-dimension 1 face.

\subsubsection{The Minkowski Conjecture on lattice tilings}\label{sss:minkowski}
This face-to-face property is in fact true if we restrict ourselves to lattice tilings of the cube, and this was another conjecture, of Minkowski \cite{minkowski1910geometrie}, that was eventually proved by Haj\'os \cite{hajos1942einfache}, a proof that was celebrated as it translated the problem to group theoretic language and proceeded to use group rings (the Fourier Transform, in our language), a first in tiling \cite{stein1994algebra}. It is interesting to see here two equivalent forms of Minkowski's Conjecture (Haj\' os' theorem) \cite{kolountzakis1998notched} stated in the language of linear forms.

\begin{theorem}\label{thm:minkowski-hajos}
If $A \in \RR^{d\times d}$ has $\det A = 1$ then there is $x \in \ZZ^d$
such that
$$
\Linf{A x} < 1,
$$
unless $A$ has an integral row.
\end{theorem}

\begin{theorem}\label{thm:minkowski-kolountzakis}
Let $B \in \RR^{d\times d}$ have $\det B = 1$ and the property that
for all $x\in\ZZ^d\setminus\Set{0}$ some coordinate
of the vector $Bx$ is a non-zero integer.
Then $B$ has an integral row.
\end{theorem}

The Keller conjecture is also easily seen to be true in low dimension but it turns out to be false when the dimension is high. The final dimension, 7, was recently settled and we now know that Keller's conjecture is true in dimensions up to 7 and false in dimension 8 or higher
\cite{brakensiek2022resolution,mackey2002cube,lagarias1992keller,perron1940}.

Let us now return to sketch Theorem \ref{thm:spectra-of-cube} and sketch its proof following \cite{kolountzakis2000packing}. Observe, given \eqref{spectrum-as-tiling}, that we have to prove the equivalence, for any $T \subseteq \RR^d$, 
\begin{equation}\label{ts-equivalence}
Q \text{ tiles with } T \iff \Abs{\ft{\one_Q}}^2 \text{ tiles with } T.
\end{equation}
We will pretend that \eqref{supports} is also a sufficient condition for tiling (if you do not like to pretend, please read the details in \cite{kolountzakis2000packing}) so that we have to prove the equivalence
\begin{equation}\label{pretend}
\supp\ft{\delta_T} \subseteq Z(\one_Q) \cup \Set{0} \iff
  \supp\ft{\delta_T} \subseteq Z(\Abs{\ft{\one_Q}}^2) \cup \Set{0}.
\end{equation}

First of all we can easily compute the Fourier zeros of $\one_Q$ to be all points in $\RR^d$ with at least one non-zero integer coordinate:
$$
Z(\one_Q) = \Set{(\xi_1, \ldots, \xi_d) \in \RR^d: \text{ some } \xi_j \in \ZZ\setminus\Set{0}}.
$$
Next, the Fourier Transform of $\Abs{\ft{\one_Q}}^2$ is the convolution $\one_Q*\one_Q$ which is non-zero precisely in $2Q$. The support of $\one_Q*\one_Q$ is the closure $\overline{2Q}$.

Observe that
\begin{equation}\label{z-contained}
Z(\one_Q) \subseteq Z(\Abs{\ft{\one_Q}}^2).
\end{equation}

Suppose now that $Q$ tiles with $T$. Then, by \eqref{supports}, we obtain
$$
\supp\ft{\delta_T} \subseteq Z(\one_Q) \cup \Set{0}.
$$
and by \eqref{z-contained} we obtain
$$
\supp\ft{\delta_T} \subseteq Z(\Abs{\ft{\one_Q}}^2) \cup \Set{0},
$$
so we have proved the $\Longrightarrow$ direction of \eqref{pretend}.

We now prove the other direction ($\Longleftarrow$) of \eqref{ts-equivalence}. Suppose that $\Abs{\ft{\one_Q}}^2$ tiles with $T$. It follows from orthogonality then that
$$
T-T \subseteq Z(\one_Q) \cup \Set{0} \subseteq (2Q)^c \cup \Set{0} = (Q-Q)^c \cup \Set{0}.
$$
It follows that the $T$ translates of $Q$ are packing (non-overlapping).

To conclude the proof in this direction we need the following interesting result \cite{lagarias2000orthonormal,kolountzakis2000packing} which, in some sense, says that tiling is a property more of the translation set (the tiling complement) rather than the tile itself.

This theorem is intuitively clear when $T$ is a periodic set but it is,
perhaps, suprising that it holds without any assumptions on
the set $T$.
Its proof is very simple.

Let us agree to say that $f+T$ \textit{is a packing} if $\sum_{t \in T} f(x-t) \le 1$ for almost every $x \in \RR^d$. Similarly we say $f+T$ \textit{is a tiling} if $\sum_{t\in T}f(x-t)=1$ for almost every $x \in \RR^d$.

\begin{theorem}\label{thm:tiling-from-packing}
If $f, g \ge 0$, $\int f(x)\, dx = \int g(x)\, dx = 1$ and
both $f+T$ and $g+T$ are packings of $\RR^d$, then
$f+T$ is a tiling if and only if $g+T$ is a tiling.
\end{theorem}

\begin{proof}
We first show that, under the assumptions of the Theorem,
\beql{supp-tiling}
\mbox{$f+T$ tiles $-\supp g$ } \Longrightarrow
\mbox{ $g+T$ tiles $-\supp f$}.
\eeq
Indeed, if $f+T$ tiles $-\supp g$ then
$$
1 = \int g(-x) \sum_{t\in T} f(x-t) \,dx = 
\sum_{t\in T} \int g(-x) f(x-t) \,dx,
$$
which, after the change of variable $y = -x+t$, gives
$$
1 = \int f(-y) \sum_{t\in T} g(y-t) \,dy.
$$
This in turn implies,
since $\sum_{t\in T}g(y-t) \le 1$ (from $g$ packing with $T$),
that $\sum_t g(y-t) = 1$ for
a.e.\ $y \in -\supp f$.

To complete the proof of the theorem, notice that if $f+T$ is a tiling
of $\RR^d$ and $a \in \RR^d$ is arbitrary then
both $f(x-a) + T$ and $g(x-a) + T$ are packings and
$f+T$ tiles $-\supp g(x-a) = -\supp g - a$.
We conclude that $g(x-a) + T$ tiles $-\supp f$, or
$g + T$ tiles $-\supp f - a$. Since $a\in\RR^d$ is arbitrary
we conclude that $g + T$ tiles $\RR^d$.

\end{proof}

See also \cite[Lemma 3.2]{greenfeld2020spectrality}.

To complete the proof of the $\Longleftarrow$ direction of \eqref{ts-equivalence} we note that we can apply Theorem \ref{thm:tiling-from-packing} with $f=\one_Q$ and $g=\Abs{\ft{\one_Q}}^2$ to obtain that $Q+T$ is a tiling from $\Abs{\ft{\one_Q}}^2+T$ being a tiling.

See also \cite{agora2018spectra} for a discrete version of the cube spectra.

\subsubsection{Maximal incomplete collections of exponentials for the cube and other spectral sets}\label{sss:maximal}

In \cite{kolountzakis2024maximality} one seeks orthogonal collections of exponentials for the cube which are maximal, in the sense that the collection cannot by increased and remain orthogonal, but at the same time they are not a spectrum (so it is a maximal and incomplete collection). This is in rough analogy with packings and tilings of the cube: it is easy to construct in any dimension a packing of translates of the cube which is maximal (no more cubes can be added without overlaps) but is also not a tiling. It turns out however that in dimensions 1 and 2 there are no incomplete, maximal collections of exponentials for the cube. Such collections do exist in dimension 3 and higher though, and some of them can be quite thin, in the sense that the set of frequencies they consist of are sets of zero density, unlike spectra which must be of positive density. Even in dimension 1 there are spectral sets which do admit a maximal, incomplete collection of exponentials.

\subsection{The ball and smooth convex bodies are not spectral}\label{ss:ball}

Smooth convex bodies such as the Euclidean ball are obviously not tiling. We shall see that they are also non-spectral \cite{iosevich2001convexbodies,kolountzakis2004distance,fuglede-ball}. The proof we will see is from \cite{kolountzakis2004distance} and it ties the spectrality problem to one of Geometric Ramsey Theory.

Suppose $\Lambda$ is a spectrum of $K$, a smooth, symmetric convex body in $\RR^d$.
It is a well known fact proved using the method of stationary phase (see, for example, \cite{sogge2017fourier})
that if $\xi$ is a zero of $\ft{\one_K}$ and $\xi\to\infty$
then
$$
\Norm{\xi}_{K^o} = \left({\pi\over 2} + {d\pi \over 4}\right) + k\pi + o(1),
    \ \ \ (\xi\to\infty),
$$
where $K^o$ is the dual body (which is also smooth)
and $k$ is an integer. This estimate is common to all proofs that $K$ is not spectral.

\begin{figure}[h]
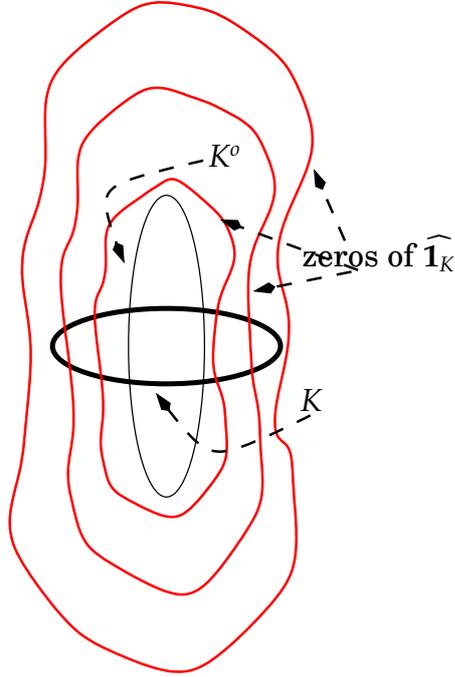
\caption{The zeros of $\ft{\one_K}$, when $K$ is a smooth convex body with positive Gaussian curvature come roughly at integer multiples (plus an offset) of the dual body $K^o$.}\label{fig:convex-zeros}
\end{figure}

Let $R>0$ be such that any zero $\xi$ of $\ft{\one_K}$, outside a cube of side
$R$ centered at the origin, satisfies
$$
\Norm{\xi}_{K^o} = \left({\pi\over 2} + {d\pi \over 4}\right) + k\pi + \theta,
    \ \ \ (k\in\ZZ,\ \Abs{\theta}<\pi/10).
$$
We also take $R$ to be large enough so as to be certain that we find
at least one $\Lambda$-point in any cube of side $R$.
We can do this since $\Lambda$ is well-distributed.

Let now the set $\Lambda'$ arise by keeping only one point of $\Lambda$
in each cube of the type $Rn + (-R/2, R/2)^d$, with $n\in\ZZ^d$ having
all its coordinates even. We keep nothing outside these cubes.
It follows that $\Lambda'$ is also a well
distributed set and that for any two distinct points $\lambda$ and $\mu$
of $\Lambda'$, $\mu$ is not contained in the cube of side $R$
centered at $\lambda$.
From the orthogonality of $\Lambda$ we obtain
that for any two distinct points $\lambda,\mu\in\Lambda'$ we have
$$
\Norm{\lambda-\mu}_{K^o} = k\pi + \theta,\ \ \ (k\in\ZZ,\ \Abs{\theta}\le\pi/5).
$$
This means that the set of $K^o$-distances defined by pairs of points of $\Lambda'$
has infinitely many gaps of length at least $3\pi/5$. We will now see that this contradicts the following result.

We call $\mu$, a probability measure,  $\epsilon$-\textit{good} if its Fourier Transform, $\ft{\mu}$, satisfies, for some finite $R>0$,
$$
\Abs{\ft{\mu}(\xi)} \le \epsilon \text{ if } \Abs{\xi} \ge R.
$$

\begin{theorem}\label{thm:good-measure}
Suppose that $A \subseteq \RR^d$, $d\ge 2$, has upper Lebesgue density at least $\epsilon>0$
and that the $0$-symmetric convex body $K$ affords $(C_d\epsilon)$-good probability
measures supported on its boundary (the constant $C_d$ depends on the dimension only).
Then $D_K(A)$ contains all positive real numbers beyond a point.
\end{theorem}

For a set $A \subseteq \RR^d$ to have upper density at least $\epsilon$ means that we can find arbitrarily large cubes in which $A$ takes up a fraction $\epsilon$ of their measure, at least. This does not preclude $A$ from having arbitrarily large gaps (large cubes where it is completely absent).

Again, by the method of stationary phase \cite{sogge2017fourier} we know that any smooth convex body with everywhere positive Gaussian curvature is such that its boundary area measure is $\epsilon$-good for all $\epsilon>0$. In other words, if $\sigma$ is the area measure on $\partial K$ then
$$
\ft{\sigma}(\xi) \to 0 \text{ as } \Abs{\xi} \to \infty.
$$
It follows that if $K$ is a smooth, symmetric convex body with everywhere positive Gaussian curvature then, since $K^o$ is also such a set, the $K^o$-distances defined by any set $A\subseteq \RR^d$ contains an interval of the form $[t_0, \infty)$ for some finite $t_0$.

An easy corollary of Theorem \ref{thm:good-measure} is that the $K^o$-distance set of a countable set $\Lambda$ with positive counting density (for some $\epsilon>0$ there are arbitrarily large balls $B$ in which the number of points of $\Lambda$ is at least $\epsilon \, \vol(B)$), itself a countable subset of $\RR^{\ge 0}$, cannot have infinitely many gaps of width $\delta>0$, no matter how small $\delta>0$ is. (Simply adjoin a $\delta/3$-ball, in the $K^o$-norm, to each point of $\Lambda$ to obtain a set $A \subseteq \RR^d$ with positive Lebesgue density and apply Theorem \ref{thm:good-measure} to it, to derive a contradiction if infinitely many gaps appear in the distance set.)

This concludes the proof that smooth, symmetric convex bodies have no spectrum.

\subsection{A few intervals}\label{ss:a-few-intervals}

In this section we encounter a few results specific to dimension one, and especially to a domain being a collection of a few intervals.

The following result \cite{laba2001twointervals} appears easier than it really is.
\begin{theorem}\label{thm:two-intervals}
The Fuglede Conjecture holds for unions of two intervals on the real line. That is, if $E \subseteq \RR$ is a union of two intervals then $E$ tiles $\RR$ if and only if $E$ has a spectrum.
\end{theorem}
Recently, the proof of Theorem \ref{thm:two-intervals} has become easier using the theory of weak tiling due to Lev and Matolcsi \cite{lev2022fuglede} (see \S \ref{s:w-tiling}), which was not available at the time Theorem \ref{thm:two-intervals} was proved in \cite{laba2001twointervals}. It is shown in \cite{laba2001twointervals} that a union $E$ of two intervals (say it has measure 1) is a tile or spectral if and only if it is of one of the following forms (see Fig. \ref{fig:fuglede-two})
\begin{enumerate}
\item
$$
E = [x, x+\frac12] \cup [x+\frac12 + k/2, x+k+1],
$$
for some nonnegative integer $k$, or
\item 
$$
E = [x, x+a] \cup [x+k+a, x+k+1],
$$
for some $0 < a < 1$ and a nonnegative integer $k$.
\end{enumerate}

\begin{figure}[h]
\ifdefined\SMART\resizebox{10cm}{!}{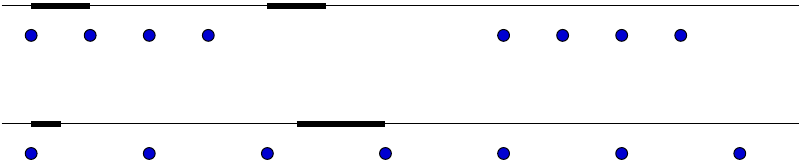}\else
\input fuglede-two.pdf_t
\fi
\caption{The two cases when a union of two intervals can tile or be spectral. In both cases the tiling complement is shown underneath the set. In the upper case we have two equal intervals if length 1/2 each, say. In this case the gap between them must be a half-integer. In the lower case the two intervals can be an arbitrary subdivision of $[0, 1]$, one of whose intervals has been moved over by an integer, thus remaining a fundamental domain of $\ZZ$ in $\RR$.
}\label{fig:fuglede-two}
\end{figure}

It is indicative of the hardness of even this problem that the naturally next case, that of three intervals, has not been resolved yet:
For a union of three intervals the ``tiling $\Longrightarrow$ spectral'' direction was established in \cite{bosefuglede}. But the other direction, ``spectral $\Longrightarrow$ tiling'', has not yet been proved in general \cite{bosefuglede,bose2014spectral}.

A single interval of course satisfies the Fuglede Conjecture as it is both a tile and spectral. Perturbations of intervals satisfy it too \cite{kolountzakis2004tiling}.
\begin{theorem}\label{thm:near-intervals}
Suppose $E\subseteq[0,L]$ is measurable with measure $1$ and
$L = 3/2-\epsilon$ for some $\epsilon>0$.

Let $\Lambda\subset\RR$ be a discrete set containing $0$.
Then
\begin{enumerate}
\item 
if $E+\Lambda =\RR$ is a tiling, it follows that $\Lambda=\ZZ$.\\
\item 
if $\Lambda$ is a spectrum of $E$, it follows that $\Lambda=\ZZ$.
\end{enumerate}
Therefore $E$ is a tile if and only if it is spectral.
\end{theorem}
That is, small perturbations of intervals that tile, are lattice tiles, hence spectral sets as well. And small perturbations of intervals that are spectral do have a lattice spectrum, so they tile too.

The situation is not so satisfactory for perturbations of higher dimensional cubes \cite{kolountzakis2004tiling}.
\begin{theorem}\label{thm:near-squares}
Let $E\subset \RR^2$ be a measurable set such that
$[0,1]^2\subset E\subset [-\epsilon,1+\epsilon]^2$
for $\epsilon>0$ small enough.
Assume that $E$ tiles $\RR^2$ by translations.

Then $E$ also admits
a tiling with a lattice $\Lambda\subseteq\RR^2$ as the translation set.
\end{theorem}
So if $E$ is a sufficiently small perturbation of the square and it tiles, then it tiles also with a lattice, so it is spectral as well. The corresponding statement for spectral perturbations of squares is still open.

\qq{
Prove the missing ``spectral $\Longrightarrow$ tiling'' direction for three intervals. If such a set is spectral then it must be weakly tiling (see \S \ref{s:w-tiling}). Perhaps, for such sets, it can be proved that weak tiling implies tiling, thus completing the proof.

The same is suggested for the ``spectral $\Longrightarrow$ tiling'' direction for perturbations of a square as in Theorem \ref{thm:near-squares}.

In both questions it may make sense to first attack the discrete problem: consider your intervals to be intervals of integers and your perturbed square to be a perturbation of the discrete square $\Set{1, 2, \ldots, N}^2$.
The essence of the problem is still there and one does not have to think about subtleties in the weak tiling measure. In the discrete context the measures are all sequences on the integers.
}

\subsection{Periodicity of the spectrum in dimension one}\label{ss:periodicity-of-spectrum}
Another similarity between tiles and spectral sets is periodicity. We know from \cite{leptin1991uniform,lagarias1996tiling,kolountzakis1996structure} that for a bounded measurable set $E \subseteq \RR$ that tiles by translation with the tiling complement $T \subseteq \RR$ the set $T$ must be periodic: there exists a positive number $t$ such that $T+t = T$. The same turns out to be true for bounded spectral subsets of $\RR$ \cite{bose2011spectrum,kolountzakis2012periodicity,iosevich2013periodicity}.
\begin{theorem}\label{thm:periodicity-of-spectrum}
Suppose that $\Lambda$ is a spectrum of $E \subseteq \RR$, a bounded measurable set of
measure 1.

Then $\Lambda$ is periodic and any period is a positive integer.
\end{theorem}
In \cite{bose2011spectrum,kolountzakis2012periodicity} Theorem \ref{thm:periodicity-of-spectrum} was proved when $E$ is a finite union of intervals.

\begin{corollary}\label{cor:multiple-tiling}
If $E$, a bounded measurable set of
measure 1, is spectral then $E$ tiles the real line at some integer
level $T$ when translated at the locations $T^{-1} \ZZ$.
\end{corollary}

\begin{proof}[Proof of Corollary \ref{cor:multiple-tiling}]
Let $\Lambda$ be a spectrum of $E$.
By Theorem \ref{thm:periodicity-of-spectrum} we know that $\Lambda$ is a periodic set and let $T$ be one of its periods:
$\Lambda+T=\Lambda$.
Then we have $\Lambda = T\ZZ + \Set{\ell_1,\ldots,\ell_T}$
(the number of elements in each period must be $T$ in order for $\Lambda$ to have
density 1, hence $T$ is an integer), and, by \eqref{diff-in-zeros} (orthogonality),
this implies that $\ft{\chi_E}(nT) = 0$ for all nonzero $n \in \ZZ$.
Hence $E$ tiles $\RR$ when translated at $T^{-1}\ZZ$ (see, e.g.\ \cite{kolountzakis1996structure})
at level $T$.
\end{proof}

Theorem \ref{thm:periodicity-of-spectrum} is not true in dimension higher than 1. For instance,
even when $E$ is as simple as a cube, it may have spectra that are not periodic as discussed in \S \ref{ss:spectra-of-cube}.

Since $E$ being spectral with spectrum $\Lambda$ is a tiling condition
$$
\Abs{\ft{\one_E}}^2 \text{  tiles when translated by } \Lambda \text{  at level } 1
$$
(assuming the measure of $E$ is 1)
one might expect that the structure theorems (see \S \ref{ss:structure}) about tilings of the real line with functions might apply. However all these theorems assume compact support for the tile (and there are non-periodic tilings when we give up compact support \cite{kolountzakis2016non}) and we do not have this here. Since $E$ is bounded, the Fourier Transform $\ft{\one_E}$ is analytic and hence its support is the entire real line. A different approach is needed.

The first observation, coming from orthogonality \eqref{diff-in-zeros}, is that for any two different points $\lambda, \mu \in \Lambda$ the difference $\Abs{\lambda-\mu}$ is bounded below by the smallest zero of $\ft{\one_E}$. The difference of successive elements of $\Lambda$ is also bounded above as we know \cite[Lemma 2.3]{kolountzakis1996structure} that $\Lambda$, as a tiling complement of an $L^1$ function, mush have density 1, hence it is impossible to have arbitrarily large gaps.

Writing
$$
\Lambda = \Set{ \cdots < \lambda_{-2} < \lambda_{-1} < \lambda_0 = 0 < \lambda_1 < \lambda_2 < \cdots}
$$
Since $\ft{\one_E}$ is analytic it has a finite number of roots in any interval, so we deduce that $\lambda_{n+1}-\lambda_n$ belongs to a finite set $\Sigma$ of positive real numbers, for all $n\in\ZZ$. From now on we view $\Lambda$ as a bi-infinite sequence of symbols from $\Sigma$ and thus as an element in the compact topological space $\Sigma^\ZZ$.

Since $\Lambda$ is a tiling complement of $\Abs{\ft{\one_E}}^2$ it follows from \eqref{supports} that
\begin{equation}\label{spectral-gap}
\supp \ft{\delta_\Lambda} \cap (0, a) = \emptyset
\end{equation}
for some $a > 0$.
Define next the subspace $X \subseteq \Sigma^\ZZ$ to consist of all sequences in $\Sigma^\ZZ$ which encode a set $\Lambda$ satisfying \eqref{spectral-gap}. Then $X$ is clearly a shift invariant set and it is not hard to see via a limiting argument that $X$ is closed in the topology of $\Sigma^\ZZ$.

The crucial lemma is:
\begin{lemma}\label{lm:half-lines}
The elements of the set $X$ are determined if we know their values on any left or right half-line.
\end{lemma}

That is, if $\Set{x_n}_{n \in \ZZ} \in X$ is such a sequence and we know the right half-sequence $\Set{x_n}_{n\ge a}$ or the left half-sequence $\Set{x_n}_{n\le a}$, for any $a \in \ZZ$, then $\Set{x_n}_{n \in \ZZ}$ is completely determined.

\begin{proof}
Suppose that $X$ is not determined by left half-lines (the argument is similar for right half-lines).
Then there are distinct $\Lambda^1, \Lambda^2 \in X$ such that
$\Lambda^1_i = \Lambda^2_i$ for all negative integers $i$.
Both $\delta_{\Lambda^1}$ and $\delta_{\Lambda^2}$ have a spectral gap at $(0,a)$ and therefore
so does their difference
$$
\mu = \delta_{\Lambda^1} - \delta_{\Lambda^2}.
$$
Notice that $\mu$ is supported in the half-line $[0, +\infty)$.
Suppose $\psi \in C^\infty(-a/10, a/10)$. It follows from the rapid decay of $\ft{\psi}$
that the measure
$$
\nu = \ft{\psi}\cdot \mu
$$
is totally bounded and
still has a spectral gap at the interval $(a/10, 9a/10)$.
But the measure $\nu$ is also supported in the half-line $[0, +\infty)$
and by the F. and M. Riesz Theorem \cite{havin1994uncertainty}
its Fourier Transform is mutually absolutely continuous with
respect to the Lebesgue measure on the line.
But this is incompatible with the vanishing
of $\ft{\nu}$ in some interval. Therefore $\nu$ must be identically $0$ and,
since $\psi \in C^\infty(-a/10, a/10)$, is otherwise arbitrary, it follows that $\mu\equiv 0 $, or
$\Lambda^1 = \Lambda^2$, a contradiction. It follows that $X$ is indeed determined by left half-lines.
\end{proof}

We also observe that determination from all half-lines implies determination by a window of some finite length.

\begin{lemma}\label{lm:diag}
Suppose $X \subseteq \Sigma^\ZZ$ is a closed, shift-invariant set which is determined by left half-lines
and by right half-lines.
Then there is a finite number $w$ such that $X$ is determined by windows of size $w$.
\end{lemma}

\begin{proof}
It is enough to show that there is a finite window size $w$ such that whenever two elements of $X$
agree on a window of size $w$ then they necessarily agree at the first index to the right of that window.
For in that case they necessarily agree at the entire right half-line to the right of the window and are
by assumption equal elements of $X$.

Assume this is not true. Then there are elements $x^n, y^n$ of $X$, $n=1,2,\ldots$,
which agree at some window of width $n$
but disagree at the first location to the right of that window. Using the shift invariance of $X$
we may assume that
$$
x^n_{-n}=y^n_{-n},\ x^n_{-n+1} = y^n_{-n+1},\ \ldots,\ x^n_{-1} = y^n_{-1}\ \ \&\ \ x^n_0 \neq y^n_0.
$$

By the compactness of the space there are $x, y \in X$ and a subsequence of $n$'s such that
$x^n \to x$ and $y^n \to y$.
By the meaning of convergence in the space $\Sigma^\ZZ$ we have that the sequences $x$ and $y$
agree for all negative indices and disagree at $0$.
This contradicts the assumption that $X$ is determined by left half-lines.
\end{proof}

We can now prove the periodicity of $x_n = \lambda_{n+1}-\lambda_n$ (which implies that $\Lambda$ is periodic). From Lemma \ref{lm:diag} $x_n$ is determined by windows of width $w$. But each such window can only be ``colored'' in finitely many ways using $\Sigma$ as the set of colors, so there are two windows with the same contents. This implies that the difference of their starting points, call it $t$, is a period of $x_n$. This number $t$ is necessarily an integer, since the density of $\Lambda$ is equal to 1 and the density of a periodic set is the number of points per period divided by the period length $t$. Since the number of points per period is integral so must be $t$.

The following is a major open question about one-dimensional spectra.

\qq{
Is it true that for a bounded spectral set $E \subseteq \RR$ its spectrum $\Lambda$ (which we may assume contains 0) is rational?

Perhaps every spectral $E$ has \textit{some} rational spectrum?

In a more basic form, we may ask if 
the spectrum of any finite set of integers (a subset of $\TT$ containing 0) is necessarily rational.
}

To the best of my knowledge a recently published proof \cite{zhou2024rationality} of the rationality of the spectrum is unfortunately incomplete.

\subsection{The product question}\label{ss:product}

Suppose we have two groups $G_1$ and $G_2$ (think of finite groups for simplicity) and two sets $A \subseteq G_1$ and $B \subseteq G_2$. It is easy to see that if $A$ tiles $G_1$ with tiling complement $T_1 \subseteq G_1$ and $B$ tiles $G_2$ with tiling complement $T_2 \subseteq G_2$ then the product $A \times B$ tiles $G_1 \times G_2$ with tiling complement $T_1 \times T_2$.

The converse is also true though not completely as obvious. If we assume that $A \times B$ tiles $G_1 \times G_2$ it follows that $A$ tiles $G_1$ and $B$ tiles $G_2$. The reason is that the intersection of any translate of $A \times B$ with the subgroup $G_1 \times \Set{0}$ is a translate of the set $A \times \Set{0}$. Thus we have $G_1 \times \Set{0}$ being tiled by copies of $A \times \Set{0}$ which is of course the same as saying that $A$ tiles $G_1$. Similarly $B$ has to tile $G_2$.

What if $A$ is spectral in $G_1$ and $B$ is spectral in $G_2$? Can we conclude that $A\times B$ is spectral in $G_1 \times G_2$? The answer is an easy yes: if $\Lambda_1 \subseteq \ft{G_1}$ is a spectrum of $A$ and $\Lambda_2 \subseteq \ft{G_2}$ is a spectrum of $B$ then $\Lambda_1 \times \Lambda_2 \subseteq \ft{G_1} \times \ft{G_2} = \ft{G_1 \times G_2}$ is a spectrum of $A \times B$.

But we do not know if the reverse implication is true.

\qq{\label{q:product}
If $A\subseteq G_1$, $B \subseteq G_2$ and $A \times B \subseteq G_1 \times G_2$ is spectral, does it follow that $A$ is spectral in $G_1$ and $B$ is spectral in $G_2$?}

Here are some partial results.
It was proved in \cite{greenfeld2016spectrality} that the answer is yes if we know that $A \subseteq \RR$ is an interval and $B \subseteq \RR^{d-1}$. This was followed by \cite{kolountzakis2016cylinders} where, in the same groups, $A$ was allowed to be a union of two intervals. It was shown then in \cite{greenfeld2020spectrality} that even if $A \subseteq \RR^2$ is a convex polygon the answer is still yes. Finally, it was proved in \cite{kolountzakis2023spectral} that if $A \subseteq \RR^m$ is a convex body and $B \subseteq \RR^n$ is bounded and $A \times B$ is spectral in $\RR^{m+n}$ then $A$ is also spectral in $\RR^m$ (but no claim is made about the spectrality of $B$).

If the answer to Question \ref{q:product} is negative this would obviously break the symmetry between tiles and spectra and one expects to be able to disprove the Fuglede Conjecture. The smallest case where this would be interesting is when $A$ and $B$ are both in $\RR$. What conclusions can we draw then if $A\times B$ is spectral in $\RR^2$ but $A$, say, is not spectral? The answer is that the Fuglede Conjecture cannot be valid in both dimensions 1 and 2. More precisely, if ``spectral $\Longrightarrow$ tiling'' were valid in dimension 2 then $A \times B$ would be a tile and therefore $A$ and $B$ would be tiles and if ``tiling $\Longrightarrow$ spectral'' were true in dimension 1  then both $A$ and $B$ would be spectral. 

We conclude that a negative answer to Question \ref{q:product} implies that either ``spectral $\Longrightarrow$ tiling'' fails in dimension 2 or ``tiling $\Longrightarrow$ spectral'' fails in dimension 1.

{\bf Note added in revision:} After this paper was written G\'abor Somlai \cite{somlai2024spectralityproductdoesimply} announced a very clever, simple example of a spectral set $A \times B \subseteq G \times G$, with $A, B \subseteq G$ so that $A$ is not spectral in $G$, thus answering Question \ref{q:product} in the negative. In this example the group $G$ is a product of several cyclic groups so this example does not disprove the Fuglede conjecture in dimension 1 or 2 as discussed in the previous paragraph. It is still open whether the spectrality of $A \times B$ implies that of $A$ \textit{or} $B$.

\section{Counterexamples to the Fuglede Conjecture}


\subsection{Spectrality in groups}\label{ss:spectrality-in-groups}

The Spectrality question can be phrased in any locally compact abelian group. Suppose $G$ is such a group and $\ft{G}$ its dual group, that is the group consisting of all continuous characters on $G$ (group homomorphisms into the multiplicative group $\CC\setminus\Set{0}$; see, for example, \cite{rudin1962groups}). Notice that we write all abelian groups additively (the only exception being the group $\Set{z \in \CC: \Abs{z} = 1}$.)

Suppose $E \subseteq G$ is a subset of $G$ with finite and non-zero Haar measure. We say that $E$ tiles $G$ with tiling complement the set $T \subseteq G$ if $\one_E * \delta_T = 1$ almost everywhere (with respect to the Haar measure) on $G$. Correspondingly, we say that $E$ is spectral and has spectrum the set $\Lambda \subseteq \ft{G}$ if the set of characters in $\Lambda$ form an orthogonal basis for the Hilbert space $L^2(E)$.

We mostly care about the ``classical'' groups $\RR^d$, $\ZZ^d$, $\TT^d$ and finite groups, and direct products that can be formed among them. For instance, a finite subset $E \subseteq \ZZ$ is spectral if we can find a finite collection of characters $e_\lambda(x) = e^{2\pi i \lambda x}$, with $\lambda \in \TT$, which forms an orthogonal basis on $L^2(E)$ (which is finite dimensional, so our question here is really a question of linear algebra). Another example is the finite group $\ZZ_N = \ZZ/(N\,\ZZ)$ whose dual group is isomorphic to itself: $\ft{\ZZ_N} \simeq \ZZ_N$. When $E \subseteq \ZZ_N$ is finite for $E$ to be spectral means to be able to find $\Lambda \subseteq \ZZ_N$ such that the characters $e_\lambda(x) = e^{2\pi i \lambda x/N}$, $\lambda \in \Lambda$, are orthogonal on $E$ and $\Abs{\Lambda} = \Abs{E}$ (this last requirement forces the completeness).

The characters $\lambda, \mu \subseteq \ft{G}$ (often we identify $\lambda$ with $e^{2\pi i \lambda x}$ with $\lambda x$ in the exponent being shorthand for $\lambda(x)$) are orthogonal on $E \subseteq G$ if and only if
$$
\one_E(\lambda-\mu) = 0.
$$
Here the Fourier Transform of $f:G\to\CC$ is the function $\ft{G}\to\CC$ given by $\ft{f}(\gamma) = \int_G f(x) \overline{\gamma(x)}\,dx$, where the integration is carried out with respect to the Haar measure on $G$. Haar measure can be normalized differently but we always choose the normalization on $\ft{G}$ that makes the following formula for Fourier inversion true:
$$
f(x) = \int_{\ft{G}} \ft{f}(\gamma) \gamma(x) \, d\gamma.
$$
For example, when $G = \ZZ_N$ and Haar measure on $G$ is the counting measure then the Haar measure on $\ft{G} \simeq \ZZ_N$ is counting measure divided by $N$, so that the formulas for the Fourier Transform and Fourier inversion are
$$
\ft{f}(\nu) = \sum_{x \in \ZZ_N} f(x) e^{-2\pi i \nu x}
\text{ and }
f(x) = \frac{1}{N} \sum_{\nu \in \ZZ_N} \ft{f}(\nu) e^{2\pi i \nu x}.
$$
By a similar argument as the one that led to \eqref{spectrum-as-tiling} we are again led to $\Lambda \subseteq \ft{G}$ being a spectrum for $E \subseteq G$ if and only if we have the tiling on $\ft{G}$
\begin{equation}\label{spectrum-as-tiling-in-groups}
\Abs{\ft{\one_E}}^2 * \one_\Lambda = \Abs{E}^2,
\end{equation}
where $\Abs{E}$ is the Haar measure of $E$.

As a warmup for what is to follow let us prove the Fuglede Conjecture in the case the group is $\ZZ_p$, with $p$ a prime, a rather easy case.

Suppose $E \subseteq \ZZ_p$ is a tile. Then obviously $\Abs{E}$ divides the size of the group, so $\Abs{E}=1$ or $\Abs{E}=p$. In other words $E$ is the whole space or just one point. In both cases it is spectral: when it is one point just choose any character as the spectrum and when it is the whole space choose all the characters as the spectrum (the characters of $G$ are always orthogonal on $G$). So we have proved the tiling $\Longrightarrow$ spectral direction of the conjecture.

For the other direction we shall need the following well known result.
\begin{theorem}\label{thm:sums-of-prime-roots-of-unity}
Suppose $p$ is a prime and $\emptyset \neq A \subseteq \Set{0, 1, 2, \ldots, p-1}$ is such that
$$
\sum_{a \in A} e^{2\pi i a/p} = 0.
$$
Then $A = \Set{0, 1, 2, \ldots, p-1}$.

In other words if the sum of some distinct $p$-th roots of unity vanishes then these must be all $p$-th roots of unity.
\end{theorem}
\begin{proof}
The minimal polynomial of $e^{2\pi i /p}$ over $\QQ$ is well known to be the cyclotomic polynomial $\Phi_p(x) = 1+x+x^2+\cdots+x^{p-1}$ so this polynomial must divide the polynomial $q(x) = \sum_{a \in A} x^a$ since it vanishes on $e^{2\pi i /p}$ by our assumption. But the degree of $q(x)$ is at most $p-1$, so it must be equal to $q-1$ and we have $q(x) = \Phi_p(x)$.
\end{proof}

\begin{corollary}\label{cor:fourier-zeros-in-zp}
If $\emptyset \neq E \subseteq \ZZ_p$ is not the whole group then $\ft{\one_E}$ has no-zeros on $\ZZ_p$.
\end{corollary}
\begin{proof}
We have $\ft{\one_E}(\nu) = \sum_{e \in E} e^{-2\pi i e \nu}$ is a sum of $p$-th roots of unity (they are all different) so they must be all $p$-th roots of unity by Theorem \ref{thm:sums-of-prime-roots-of-unity}.
\end{proof}

Suppose now that $\emptyset \neq E \subseteq G$ is spectral with spectrum $\Lambda \subseteq \ZZ_p$. Since we must have $\ft{\one_E}(\lambda-\mu)=0$ for all $\lambda \neq \mu \in \Lambda$ it follows from Corollary \ref{cor:fourier-zeros-in-zp} that $E$ is the whole group or that we cannot find two different points in $\Lambda$, so that $E$ is either the whole group or a single point. In both cases it tiles.

\subsection{Failure of the ``spectral $\Longrightarrow$ tiling'' direction}\label{ss:st-fails}

In 2003 T. Tao \cite{tao2004fuglede} surprised the community of researchers working on the Fuglede Conjecture by disproving the ``spectral $\Longrightarrow$ tiling'' with a very easy argument. Until that point all the partial results regarding the Fuglede Conjecture were in the direction of supporting it.

Tao first disproved the conjecture in a finite group then lifted the example to the group $\RR^5$. In \cite{matolcsi2005fuglede4dim,kolountzakis2006hadamard} the dimension $d$ was reduced to 3 (where it still stands).

Let us start with a slightly weaker result (because the number of factors in the group, 12, is larger than the promised 5).
\begin{theorem}\label{thm:st-fails-tao}
The ``spectral $\Longrightarrow$ tiling'' direction of the Fuglede Conjecture fails in the group $\ZZ_2^{12}$.
\end{theorem}
\begin{proof}
Take $E = \Set{e_1, e_2, \dots, e_{12}}$ to consist of all axis vectors: $e_j$ is all zeros except at the $j$-th position where it has a 1. Then $\Abs{E}=12$ which does not divide the size of the group $\Abs{\ZZ_2^{12}} = 2^{12}$, so $E$ cannot tile the group. We show that $E$ is spectral.

As is the case with all finite groups we have $\ft{\ZZ_2^{12}} \simeq \ZZ_2^{12}$. The characters on $\ZZ_2^{12}$ are the functions
$$
e_\gamma(x) = (-1)^{\sum_{j=1}^{12} \gamma_j x_j} = \prod_{1}^{12} (-1)^{\gamma_j} (-1)^{x_j},
$$
where $\gamma \in \ZZ_2^{12}$. To have a spectrum of $E$ means that we have a collection of $\gamma$s such that the matrix
$$
M_{j, k} = (-1)^{\gamma_{j, k}} (-1)^{e_{j, k}},\ \ \ (j, k = 1, 2, \ldots, 12)
$$
has orthogonal rows. This is a matrix with entries $\pm 1$ and orthogonal matrices of this type are called Hadamard matrices. Such a $12 \times 12$ matrix does exist \cite{tao2004fuglede} (in general it is not known for which dimensions Hadamard matrices exist). This $12\times 12$ Hadamard matrix gives us the elements of the spectrum of $E$.
\end{proof}

Working similarly in the group $G = \ZZ_3^6$ Tao obtained a counterexample in this group. Using again $E = \Set{e_1, \ldots, e_6}$ (again, 6 does not divide the order of the group, which is $3^6$) a spectrum is found whenever we can find a $6 \times 6$ matrix $M$ with elements in $\Set{0, 1, 2}$ such that the matrix
$$
T_{j, k} = (e^{2\pi i /3})^{M_{j,k}}
$$
is orthogonal. Such a matrix was found by computation \cite{tao2004fuglede} and is
$$
M = \left(
\begin{array}{cccccc}
0&0&0&0&0&0\\
0&0&1&1&2&2\\
0&1&0&2&2&1\\ 
0&1&2&0&1&2\\
0&2&2&1&0&1\\ 
0&2&1&2&1&0
\end{array}
\right).
$$
To obtain a counterexample in the even smaller group $\ZZ_3^5$ we observe that $E$ is contained in a coset of the subgroup $\Set{x: \sum_{k=1}^6 x_j = 0} \simeq \ZZ_3^5$ in $\ZZ_3^6$. Translating it to 0 does not change its tiling or spectrality properties so we can take $E$ to be a subgroup of $\ZZ_3^5$ instead of $\ZZ_3^6$.

This last step demands some more explanation. Whenever we have a set $E \subseteq G$ and we study its tiling properties (``does it tile $G$?'') we can always work in the subgroup $G'$ generated by $E$. If $E$ tiles $G'$ then we can translate this tiling to every coset of $G'$ in $G$ and obtain a tiling of $G$. Conversely, if $E$ tiles $G$ then the translates of $E$ that participate in the tiling also tile $G'$ as it is impossible for a translate of $E$ to intersect $G'$ without being contained in it. Thus $E$ tiles $G'$ if and only if $E$ tiles $G$.

Regarding spectrality of $E \subseteq G' \subseteq G$, if $E$ is spectral in $G$ (using characters of $G$) it is also spectral in $G'$ as every character of $G$ is a character of $G'$ when restricted to $G'$. And if $E$ is spectral in $G'$ (using characters of $G'$) it is also spectral on $G$ as every character on $G'$ can be extended to a character of $G$.

Thus when thinking about tiling and spectrality we can always assume, if that suits us, that $E$ generates the group.

Back to Tao's example in $\ZZ_3^6$ now: properly translating $E$ can place it into a subgroup of $\ZZ_3^6$ which is isomorphic to $\ZZ_3^5$. Viewed in this group $E$ is spectral but not a tile, by the discussion above.

Having a counterexample $E$ to ``spectral $\Longrightarrow$ tiling'' in $\ZZ_3^5$ the next step is to lift it to a counterexample $E'$ in $\ZZ^5$. (This is described in \cite{tao2004fuglede}.) The set $E'$ then is lifted to a counterexample $E''$ in $\RR^5$ by just attaching a unit cube to each point in $E'$. This last step is very easy to verify that it preserves both the spectrality of $E'$ and its non-tiling character.

The dimension 5 was further reduced to 3 in \cite{matolcsi2005fuglede4dim,kolountzakis2006hadamard}. The main innovation there is that in place of the usual Hadamard matrices one considers orthogonal matrices whose entries are complex numbers of modulus one. In \cite{diaz2016fourier} useful parametrizations of such families are given and one of these was used to reduce the dimension to 3. The extension to $\ZZ^3$ and $\RR^3$ works in the same way as in \cite{tao2004fuglede}.

\subsection{Failure of the ``tiling $\Longrightarrow$ spectral'' direction}\label{ss:ts-fails}

In the disproof of the ``spectral $\Longrightarrow$ tiling'' direction in a finite group there was the great advantage of having an easy test for non-tiling: if the size of $E \subseteq G$ (assume $G$ is a finite group) does not divide the size of $G$ then $E$ does not tile $G$. Unfortunately no such easy criterion is known with which to disprove spectrality.

We will first disprove a stronger conjecture, the Universal Spectrum Conjecture of Lagarias and Wang \cite{lagarias1997spectral}.

\begin{quotation}{\bf Universal Spectrum Conjecture}
If $G$ finite and $E$ tiles $G$ then all its tiling complements
have a {\em common} spectrum.
\end{quotation}

Take $G = \ZZ_6^5$ and $E = \Set{0, e_1, e_2, \ldots, e_5} \subseteq G$. Define $v=(1, 2, 3, 4, 5)$ and the associated homomorphism $\phi:G \to \ZZ_6$ defined by
$$
\phi(x) = v \cdot x \bmod 6.
$$
Then $\phi$ is one to one on $E$ and $\phi(E) = \ZZ_6$ hence $E$ tiles $G$ with $T = \ker\phi$. We can do this with any permutation of the coordinates of $v$ and get several tiling complements of $E$ in $G$.

Assume now that $E$ has the set $L$ as a universal spectrum, i.e., $L$ is a spectrum for any tiling complement $T$ of $E$ in $G$. This implies that $\Abs{L}=6^4$. It also means that
$$
L - L \subseteq Z(\one_T) \cup {0}.
$$
By the definition of $T$ as the kernel of $\phi$ we have, with $\zeta_6 = e^{2\pi i /6}$,
$$
\ft{\one_T}(\lambda v) = \sum_{t\in T} \zeta_6^{\lambda v \cdot t} = \Abs{T},\ \ \ (\lambda \in \ZZ_6).
$$
With $\lambda=2$ this implies that $2v = (2, 4, 0, 2, 4) \notin L-L$ and, permuting the coordinates of $v$, we get that any vector which is a permutation of $(0, 2, 2, 4, 4)$ is not in $L-L$.

Define the matrix
\begin{equation}\label{K-matrix}
K = \left(
\begin{array}{cccccc}
0&0&2&2&4&4\\
0&2&0&4&4&2\\
0&2&4&0&2&4\\
0&4&4&2&0&2\\
0&4&2&4&2&0
\end{array}
\right)
\end{equation}
and observe that the differences of all its columns (call $K\subseteq\ZZ_6^5$ also the set of columns of the matrix $K$) are permutations of $(0, 2, 2, 4, 4)$, so that
$$
(K-K) \cap (L-L) = \Set{0}.
$$
Hence $K+L$ is a packing. But $\Abs{K}=6$ and $\Abs{L}=6^4$ so that this is actually a tiling of $G$. But this cannot be, since $K$ is contained in the subgroup of $\ZZ_6^5$ of vectors with even coordinates and $K$ cannot tile this subgroup as $\Abs{K}=6$ and the subgroup has size $3^5$.

We have thus disproved the Universal Spectrum Conjecture. Next we will prove that the ``tiling $\Longrightarrow$ spectral'' direction fails in the group
$$
G_2 = \ZZ_6^5 \times \ZZ_N,
$$
where $N$ is the number of tiling complements of $E$ in our disproof of the Universal Spectrum Conjecture above. Call these tiling complements of $E$ in $\ZZ_5^6$ by the names $T_0, T_1, \ldots, T_{N-1}$.

Denote by $\widetilde{x} = (x_1, \ldots, x_5, 0)$ the embedding $\ZZ_6^5 \to G_2$. Define the set
$$
\Gamma = \bigcup_{j=0}^{N-1} \left(\widetilde{T_j} + (0,0,\ldots,0,j)\right).
$$
(See Fig.\ \ref{fig:layers}.)
Obviously $\Gamma + \widetilde{E} = G_2$ is a tiling.

\begin{figure}[h]
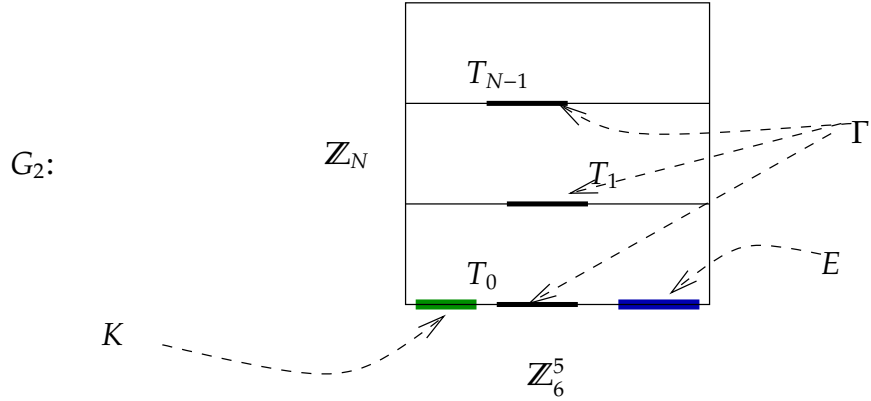
\caption{The construction of the set $\Gamma$ in $G_2$.
We move each tiling complement $T_j$ of $E$ in $\ZZ_6^5$ to a different coset of $\ZZ_6^5$ in $G_2$.
}\label{fig:layers}
\end{figure}

Assume that $S$ is a spectrum of $\Gamma$ in $G_2$ so that
$$
S-S \subseteq Z(\one_\Gamma) \cup \Set{0}.
$$
Let $k$ be a difference of two columns of $K$. Then
\begin{equation}\label{t0}
\ft{\one_\Gamma}\left(\widetilde{k}\right)  =
 \sum_{j=0}^{N-1} \left(\one_{\widetilde{T_j}+(0,0,\ldots,0,j)}\right)^\wedge \left(\widetilde{k}\right) > 0,
\end{equation}
since
\begin{itemize}
\item
The Fourier Transforms of the subgroups $\widetilde{T_j}$ are everywhere nonnegative (true for all subgroups).
\item 
The phase factor introduced to the Fourier Transform with the translation by $(0, \ldots, 0, j)$ has no effect as the last coordinate of $\widetilde{k}$ is 0.
\item 
For any such $k$, by the previous disproof of the Universal Spectrum Conjecture, for some $j=0, 1, \ldots, N-1$,
$$
\left(\one_{\widetilde{T_j}+(0,0,\ldots,0,j)}\right)^\wedge(\widetilde{k}) > 0.
$$
\end{itemize}
So $\widetilde{K}-\widetilde{K} \subseteq Z(\one_\Gamma)^c$, or
$(\widetilde{K}-\widetilde{K}) \cap (S-S) = \Set{0}$ and it follows that  $\widetilde{K}+S = G_2$ is a tiling since $\Abs{\widetilde{K}}=6, \Abs{S} = N\cdot 6^4$. But this is a contradiction since $\widetilde{K}$ is not a tile
(from the previous proof).

We have thus disproved the ``tiling $\Longrightarrow$ spectral'' conjecture for the group $G_2 = \ZZ_6^5 \times \ZZ_N$. Notice now that we can increase $N$ as much as we wish by repeating one of the $T_j$ sets more than once. The proof remains valid. Increasing thus $N$ to be coprime to 6 and remembering that the product group $\ZZ_a \times \ZZ_b$ is isomorphic to the cyclic group $\ZZ_{ab}$ if $a$ and $b$ are coprime, we see that we have an example in the group $\ZZ_6^4 \times \ZZ_{6N}$, which has only 5 factors, not 6 as before. This makes a difference in the final dimension we obtain for our example in $\ZZ^d$ and $\RR^d$.

\begin{theorem}\label{thm:lifting}
Any counterexample to the ``tiling $\Longrightarrow$ spectral'' direction of the Fuglede Conjecture in a product group $\ZZ_{n_1} \times \cdots \times \ZZ_{n_d}$ gives rise to a counterexample in $\ZZ^d$ and in $\RR^d$.
\end{theorem}

\begin{proof}
Assume $A \subseteq G = \ZZ_{n_1}\times\cdots\times\ZZ_{n_d}$ is a non-spectral tile and write 
$$
T = \Set{0,n_1,2n_1,\ldots,(k-1)n_1}\times\cdots\times
                        \Set{0,n_d,2n_d,\ldots,(k-1)n_d}
$$
for the $k \times \cdots \times k$ grid spaced by $n_1 \times \cdots \times n_d$.
Finally, define $A(k) = A+T$ (see Fig.\ \ref{fig:a-k}).

\begin{figure}[h]
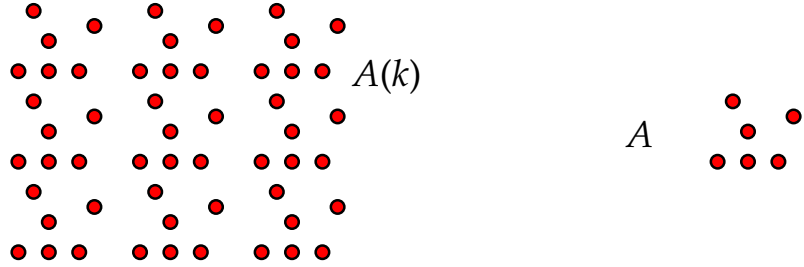
\caption{If $A$ is a non-spectral tile in $\ZZ_{n_1}\times\cdots\times\ZZ_{n_d}$ then the set $A(k) = A + \Set{0,n_1,2n_1,\ldots,(k-1)n_1}\times\cdots\times      \Set{0,n_d,2n_d,\ldots,(k-1)n_d}$ is a non-spectral tile in $\ZZ^d$ if $k$ is sufficiently large.
}\label{fig:a-k}
\end{figure}

The set $A(k)$ clearly tiles $\ZZ^d$ for any value of $k$ because $A$ tiles $\ZZ_{n_1}\times\cdots\times\ZZ_{n_d}$ (tilings of this group translate to periodic tilings of $\ZZ^d$ with period lattice $n_1\ZZ \times \cdots \times n_d\ZZ$). We will show that it is not spectral in $\ZZ^d$. Suppose $S \subseteq \TT^d$ is a spectrum of $A(k)$.

\begin{figure}[h]
\ifdefined\SMART\resizebox{10cm}{!}{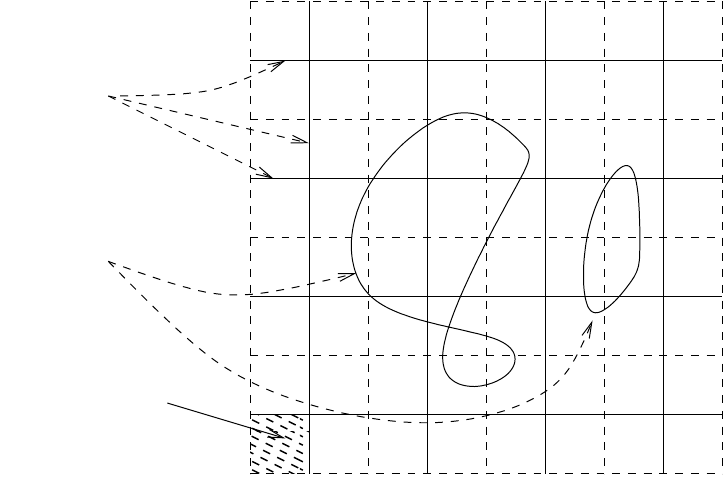}\else
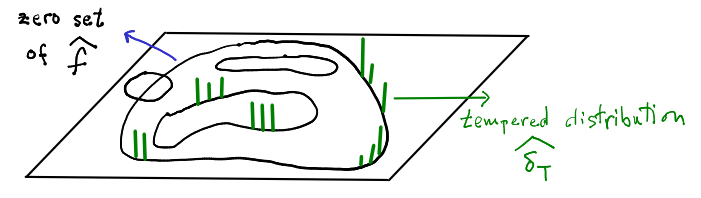
\fi
\caption{The zero set of $\ft{\one_{A(k)}}$ consists of the zero set of the trigonometric polynomial $\ft{\one_A}$ with a set of affine subspaces parallel to the axes from from $\ft{\one_T}$.}\label{fig:zero-set}
\end{figure}

We have $\one_{A(k)} = \one_A*\one_T$ so we have $\ft{\one_{A(k)}} = \ft{\one_A}\cdot\ft{\one_T}$ and for the zero sets we have
$$
Z(\one_{A(k)}) = Z(\one_A) \cup Z(\one_T).
$$
See Fig.\ \ref{fig:zero-set}.

An easy calculation shows that
$$
Z(\one_T) = \Set{\xi \in \TT^d:\ \xi_j = \frac{\nu}{k n_j} \text{ for some $j$ and integer $\nu$, where $k$ does not divide $\nu$}}.
$$
Define the rectangle
$$
Q= \left[0,\frac{1}{kn_1}\right)\times\cdots\times\left[0,\frac{1}{kn_d}\right)
$$
and the subgroup $H$ of $\TT^d$ (independent of the parameter $k$) consisting of all points $\xi$ in $\TT^d$ all of whose coordinates $\xi_j$ are multiples of $1/n_j$ (for all $j$). We can view $H$ as the dual group of the product $\ZZ_{n_1} \times \cdots \times \ZZ_{n_d}$.

\begin{figure}[h]
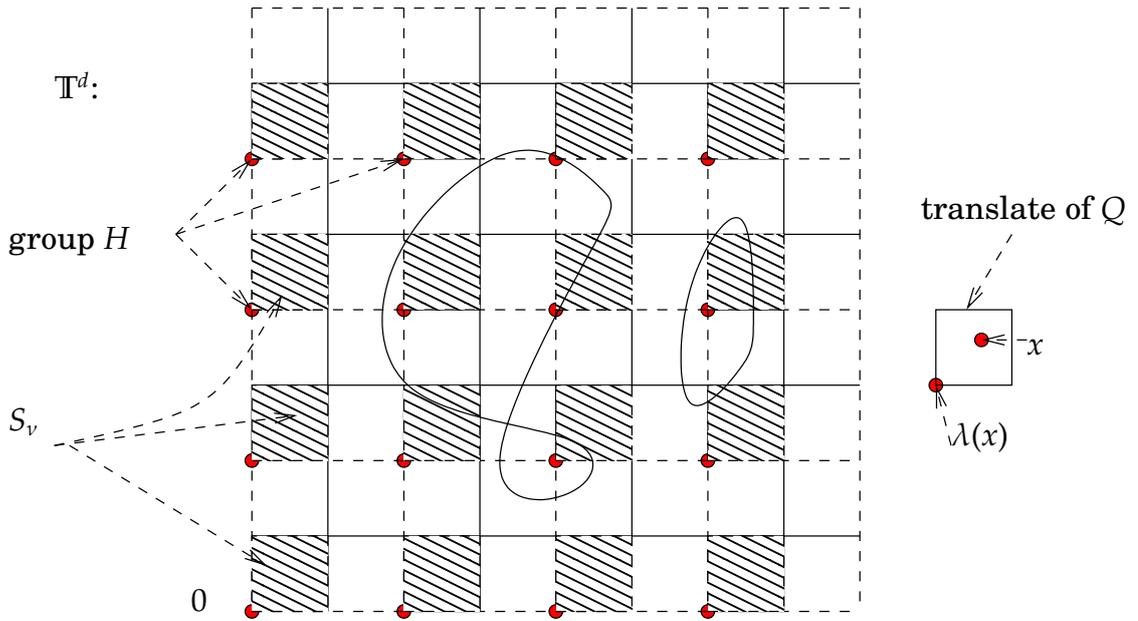
\caption{The various sets defined in the lifting of non-spectrality from a finite group to $\ZZ^d$.}\label{fig:zs-proof}
\end{figure}

Observe next that $H + (Q-Q)$ does not intersect $Z(\one_T)$ and take $k$ large enough (which makes $Q$ small enough) to ensure:
$$
\mbox{$h\in H$ with $\ft{\one_A}(h) \neq 0$} \Rightarrow
 \mbox{$h+(Q-Q)$ does not meet $Z(\one_A)$}.
$$
Partition now the spectrum $S$ into the sets $S_\nu$ as follows. Define
$$
S_\nu = S \cap \left(H+Q+\left(\frac{\nu_1}{kn_1},\ldots,\frac{\nu_d}{kn_d}\right)\right),
$$
for $\nu \in \Set{0,\ldots,k-1}^d$.
Since $\Abs{S} = \Abs{A} k^d$ it follows that there exists some $S_\mu$ with
$$
\Abs{S_\mu} \ge \Abs{A}.
$$
As $Q-Q$ does not meet $Z_{A(k)}$ we cannot have more than one point of $S$ in any translate of $Q$.

For $x \in \TT^d$ write $\lambda(x)$ for the unique point whose $j$-th coordinate is an integer multiple of $1/(k n_j)$ (for all $j$) and is such that $x \in \lambda(x)+Q$ (the lower left corner of the rectangle to which $x$ belongs).
It follows that for any $x, y \in \TT^d$ we have
$$
x-y \in \lambda(x)-\lambda(y) + Q-Q.
$$
If $x, y \in S_\mu$ then $\lambda(x)-\lambda(y) \in H \cap Z(\one_A)$. Refer to Fig.\ \ref{fig:zs-proof}.

Define now $\Lambda = \Set{\lambda(x):\ x \in S_\mu}$ and translate $\Lambda$ so that it contains 0 and so that $\Lambda \subseteq H$. Since $\Lambda-\Lambda \subseteq Z(\one_A)\cup\Set{0}$ and $\Abs{\Lambda} \ge \Abs{A}$ we obtain that $\Lambda$ is a spectrum of $A$ in the group $G$, a contradiction. 

\end{proof}

Suppose now that we have a non-spectral tile $A$ in $\ZZ^d$, a finite set. Let us see how we can construct a non-spectral tile in $\RR^d$. Let $Q = [0, 1)^d$ and define the set $E = A+Q \subseteq \RR^d$. Obviously $E$ tiles $\RR^d$ with the same tiling complement it has in $\ZZ^d$. Let us show that it is still not spectral in $\RR^d$.

Suppose $E$ has a spectrum $S \subseteq \RR^d$. It follows that $\dens S = \Abs{A}$ so there exists a $k \in \ZZ^d$ so that $\Lambda = S \cap (k+Q)$ has at least $\Abs{A}$ points. Viewing $Q$ as $\TT^d$ we show that $\Lambda$ is a spectrum of $A$.

Observe first that $Z(\one_E) = Z(\one_A) \cup Z(\one_Q)$ where the set $Z(\one_A)$ is a $\ZZ^d$-periodic set in $\RR^d$. But $\Lambda-\Lambda \subseteq Q-Q$ and $Q-Q$ does not intersect $Z(\one_Q)$, which consists of all points with at least one non-zero integer coordinate. It follows that $\Lambda-\Lambda \subseteq Z(\one_A)$ and $\Lambda$ is a spectrum of $A \subseteq \ZZ^d$, a contradiction, so $E$ is not spectral in $\RR^d$.

So far we have produced an example contradicting ``tiling $\Longrightarrow$ spectral'' in $\RR^d$ for $d \ge 5$. In \cite{farkas2006onfuglede,farkas2006tiles} the dimension was eventually reduced to 3, so that the Fuglede Conjecture, in its generality, remains open in both directions in dimensions 1 and 2 only.

\subsection{Connected counterexamples}\label{ss:connected}

Another recent development \cite{greenfeld2023tiling} is the construction of counterexamples for both directions of the Fuglede Conjecture (as well as for aperiodic translational tiles \cite{greenfeld2024counterexample}) which are connected. This is achieved by using the existing counterexamples for each direction to construct connected counterexamples in higher dimension.

For the ``spectral $\Longrightarrow$ tiling'' counterexample the dimension has to be increased by 2, thus giving connected counterexamples for this direction in dimension 5.

For the ``tiling $\Longrightarrow$ spectral'' direction though the dimension has to be increased by a number that depends on the counterexample used as a seed to a certain iterative procedure (which increases the dimension at every step) that eventually produces a connected counterexample. Though we know that this process finishes in a finite number of steps we do not know what the final dimension of the counterexample will be apart from the fact that it exists.

\qq{
Improve the process \cite{greenfeld2023tiling} that constructs a connected counterexample to the ``tiling $\Longrightarrow$ spectral'' direction of the Fuglede Conjecture. More precisely, reduce the price one has to pay in dimension increase to a constant number that does not depend on the seed counterexample that initiates the connectification process in \cite{greenfeld2023tiling}.
}

\section{Orthogonal exponentials on the disk}


In \cite{fuglede1974operators} it was claimed that the disk in the plane (and the Euclidean ball in $\RR^d$) is not a spectral set. According to the Fuglede Conjecture this is as it should be.
A proof appeared in \cite{iosevich-katz-pedersen}.
Later it was proved in \cite{fuglede-ball,iosevich2003combinatorial}
that any orthogonal set of exponentials for the ball must necessarily be finite but 
it is still unknown if there is a uniform bound for the size of each
orthogonal set. It is still a possibility that there are arbitrarily large orthogonal sets of exponentials for the ball and proving a uniform upper bound is probably very hard as it appears to depend on the existence or not of algebraic relations among the roots of the Bessel function $J_1$.

In the direction of showing upper bounds for orthogonal sets of exponentials it
was proved in \cite{iosevich-jaming} that if $\Lambda$ is a set of orthogonal exponentials for the ball then $\Abs{\Lambda \cap [-R,R]^d} = O(R)$, with the implicit constant independent of $\Lambda$.
Completeness would of course require that $\Abs{\Lambda \cap [-R,R]^d} \gtrsim R^d$.

Here we will describe the argument in \cite{iosevich2013size} which led to the bound
\begin{equation}\label{bound-2-3}
\Abs{\Lambda \cap [-R,R]^d} = O(R^{2/3}).
\end{equation}

Suppose $D$ is the unit disk in the plane, centered at 0 and $\Lambda$ is an orthogonal set of exponentials for $D$. Write $N = \Abs{\Lambda \cap [-R, R]^d}$. We seek upper bounds on $N$ in terms of $R$, for large $R$.

By the orthogonality of $\Lambda$ we have $\Lambda-\Lambda \subseteq Z(\one_D) \cup \Set{0}$ so we need information for the zero set $Z(\one_D)$. Clearly this is a radial set in the plane, a collection of concentric circles centered at the origin and let us call the radii of these circles
$$
0 < r_1 < r_2 < \cdots.
$$
There are some well known estimates for these numbers (which happen to be the zeros of the Bessel function $J_1(2\pi r)$) \cite{abramowitz-stegun}:
\begin{equation}\label{rn-1}
r_n = \frac{n}{2} + \frac18 + \frac{K_1}{2\pi(n\pi + \pi/4)} + O(\frac{1}{n^3})
\end{equation}
where $K_1$ is an absolute constant. It follows from this that if $0<m-n\le K$ and $m, n \to +\infty$ then
\begin{equation}\label{rn-diff}
r_m - r_n = \frac{m-n}{2} + O((r_m-r_n)r_n^{-2}).
\end{equation}
It is the positive offset $1/8$ in \eqref{rn-1} that makes the following lemma possible. It roughly says that three points of the spectrum cannot be approximately on the same line, if they are far apart.

\begin{lemma}\label{lm:angle}
There are constants $R_0, C>0$ such that whenever
$a, b, c \in \RR^2$ are orthogonal for the unit disk, with
$\Abs{a-c}, \Abs{b-c}, \Abs{a-b} \ge R \ge R_0$ then the two largest
angles of the triangle $abc$ (as well as all its external angles) are
\begin{equation}\label{angle-bound}
\ge \frac{C}{R^{1/2}}.
\end{equation}
\end{lemma}

\begin{proof}
Assume without loss of generality that $R=\Abs{a-c}\le\Abs{b-c}\le\Abs{a-b}$
(see Fig.\ \ref{fig:triangle}).
Writing $\theta = \ft{bac}$ for the second largest angle and $T=\Abs{a-b}$ we have

\begin{figure}[h]
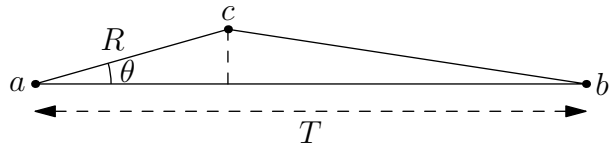

\begin{center}
\begin{asy}
import graph;
import markers;
size(8cm);
real off=0.3;
pair    a=(-1,0), b=(1,0), c=(-off,0.2);
pair    d=(-off,0);
dot(a); dot(b); dot(c);
draw(a -- b -- c -- a);
draw(c -- d, dashed);
draw("$T$", (a-(0,0.1)) -- (b-(0,0.1)), dashed, Arrows);
markangle("$\theta$", b, a, c);
label("$a$", a, W);
label("$b$", b, E);
label("$c$", c, N);
//label("$d$", d, NE);
label("$R$", 0.5*(a+c), NW);
\end{asy}

\caption{Three points orthogonal for the unit disk cannot be approximately on a straight line.}\label{fig:triangle}
\end{center}
\end{figure}

\begin{align*}
\Abs{b-c} &= \sqrt{(T-R\cos\theta)^2 + R^2 \sin^2\theta}\\
 &= \sqrt{(T-R)^2 + 2T R(1-\cos\theta)}
\end{align*}
from which we get
\beql{tmp-1}
\Abs{b-c}-(T-R) = \frac{2T R(1-\cos\theta)}{T-R+\Abs{b-c}} =
 \frac{2R(1-\cos\theta)}{1-\frac{R}{T} + \frac{\Abs{b-c}}{T}} \le 2R(1-\cos\theta) \le R\theta^2.
\eeq
From \eqref{rn-1} it follows that
as $R\to\infty$ the quantities $\Abs{a-b}, \Abs{b-c}, \Abs{a-c}$ are all of the form
$$
\frac{k}{2} + \frac{1}{8} + o(1),\ \ \ \mbox{for some integer $k$}.
$$
It follows that $\Abs{b-c}-(T-R) = \frac{k}{2} + \frac{1}{8} + o(1)$, for some integer $k\ge 0$.
This, together with \eqref{tmp-1}, implies that $\frac{k}{2} + \frac{1}{8} + o(1) \le R \theta^2$ which
gives us the required inequality with constant $C$ arbitrarily close to $\sqrt{1/8}$ when $R$ is large.

\end{proof}

The following is almost immediate from Lemma \ref{lm:angle}. See Fig.\ \ref{fig:strip}.

\begin{corollary}\label{cor:strip}
There is a constant $C'>0$ such that whenever
$a, b, c \in \RR^2$ belong to a $R_0$-separated orthogonal set
for the unit disk and their pairwise distances are at least $L$
then they cannot all belong to a strip of width $C' L^{1/2}$.
\end{corollary}

\begin{figure}[h]
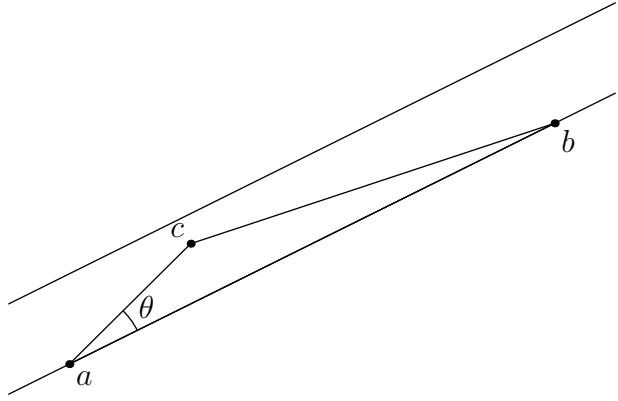

\begin{center}
\begin{asy}
import graph;
import markers;
size(8cm);
pair    A=(0,0), u=(10,5);
pair    a=A+0.1*u, b=A+0.9*u, c=A+(0,1)+0.3*u;
draw(A -- A+u);
draw(A+(0,1.5) -- A+(0,1.5)+u);
draw(a -- b -- c -- a);
dot(a); dot(b); dot(c);
label("$a$", a, SE);
label("$b$", b, SE);
label("$c$", c, NW);
markangle("$\theta$", b, a, c);
\end{asy}

\caption{Three points of the spectrum that are distance at least $L$ from one another cannot belong to the same strip of width $\sim\sqrt{L}$.}\label{fig:strip}
\end{center}
\end{figure}

From Corollary \ref{cor:strip} we get:

\begin{corollary}\label{cor:smallest-gap}
Suppose $\Lambda \subseteq \RR^2$ is a $R_0$-separated set of orthogonal exponentials for the unit disk, $R>0$ and let
$$
\Delta = \inf\Set{\Abs{\lambda-\mu}:\ \lambda, \mu \in \Lambda\cap[-R,R]^2}.
$$
Then
\begin{equation}\label{bound-wrt-gap}
\Abs{\Lambda \cap [-R, R]^2} \le C \frac{R}{\Delta^{1/2}},
\end{equation}
for some constant $C>0$.
\end{corollary}

\begin{proof}
Cover $[-R,R]^2$ by $O(R/\Delta^{1/2})$ vertical strips of width $c\Delta^{1/2}$, for small $c>0$.
From Corollary \ref{cor:strip} each of these contains at most two points of $\Lambda$.

\end{proof}

With $\Delta$ defined in Corollary \ref{cor:smallest-gap} we have $\Delta \ge c>0$ (for any spectrum the points are separated by a positive constant at least) so we already have from Corollary \ref{cor:smallest-gap} the bound $\Abs{\Lambda} = O(R)$ or \cite{iosevich-jaming}.

The main theorem of \cite{iosevich2013size} is the following.
\begin{theorem}\label{thm:disk-bound}
There are constants $C_1, C_2$ such that
whenever $\Lambda \subseteq \RR^2$ is an orthogonal set of exponentials for the unit disk in the plane and
$$
t = \inf\Set{\Abs{\lambda-\mu}: \lambda, \mu \in \Lambda, \lambda\neq\mu}
$$
then $\Abs{\Lambda} \le C_1 t$.

Furthermore, $\Abs{\Lambda \cap [-R,R]^2} \le C_2 R^{2/3}$ for all $R\ge 1$.
\end{theorem}

Let us sketch the proof of Theorem \ref{thm:disk-bound}. For the details in the various delicate estimates see \cite{iosevich2013size}.

Write $\Delta = t/2$ and after an appropriate rigid motion of $\Lambda$, to which we are entitled, we can assume that two points of $\Lambda$ are
$$
V = (\Delta, 0) \text{ and } -V = (-\Delta, 0).
$$
Let $\lambda$ be any third point in $\Lambda$, say in the first quadrant and consider the hyerbola with foci at $\pm V$ and going through $\Lambda$. Call $a(\lambda)$ the point where this hyperbola intersects the $x$-axis. See Fig.\ \ref{fig:hyperbola}.
 
\begin{figure}[h]
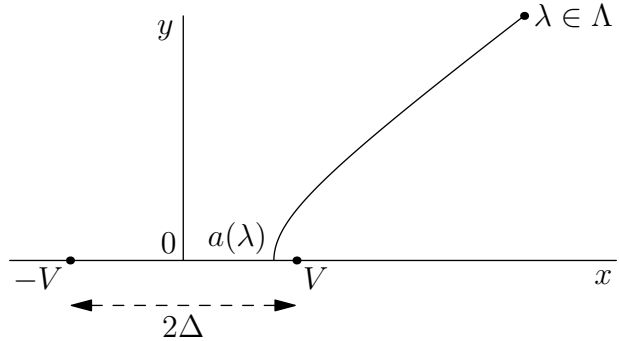

\begin{center}
\begin{asy}
import graph;
import contour;
import markers;
size(8cm);

real    a=8.0, c=10.0;
real    b, factor=3;
b = sqrt(c^2-a^2);

real[]  v={1.0};

real f(real x, real y) {return x^2/a^2 - y^2/b^2;}

real wdt=factor*c;
real hgt=(b/a)*factor*c;
real t=2.0;
real x=a*cosh(t);
real y=b*sinh(t);
//fill((0,0)--(wdt+c,0)--(wdt,hgt)--cycle, mediumgray);
dot("$\lambda\in\Lambda$", (x,y));
draw(contour(f, (0, 0), (factor*c, factor*c), v));
//draw("$L$", (0,0) -- (wdt, hgt), NW);
//draw("$M$", (c,0) -- (wdt+c, hgt));
draw("$2\Delta$",(-c,-0.4*c) -- (c,-0.4*c), dashed, Arrows, PenMargins);
//markangle("$c\Delta^{-1/2}$", (2*c,0), (c,0), (wdt+c, hgt));

xaxis("$x$");
dot("$-V$", (-c,0), SW);
dot("$V$", (c,0), SE);
label("$a(\lambda)$", (a,0), NW);
label("0", (0,0), NW);
yaxis("$y$",0,y);
\end{asy}

\caption{Take the hyperbola with foci at $\pm V$ that goes through another point $\lambda \in \Lambda$. Call $a(\lambda)$ the point where the hyperbola intersects the $x$-axis.}\label{fig:hyperbola}
\end{center}
\end{figure}

The quantity $2 a(\lambda)$ is the difference of the distances of $\lambda$ from the foci of the hyperbola and from the asymptotics \eqref{rn-diff} we obtain
\begin{equation}\label{2al}
2a(\lambda) = \Abs{\lambda+V} - \Abs{\lambda-V} = \frac{k}{2}+O(\Delta \Abs{\lambda}^{-2}),
\end{equation}
for some integer $k$.

Define the hyperbolas $H_k$ as the locus of points $p$ in the plane satisfying
$$
\Abs{p+V} - \Abs{p-V} = \frac{k}{2},\ \ \ k=0, 1, \ldots, \Floor{4\Delta}.
$$
There are $O(\Delta)$ such hyperbolas and it is important that each $\lambda$ is ``near'' one such hyperbola.

\begin{figure}[h]
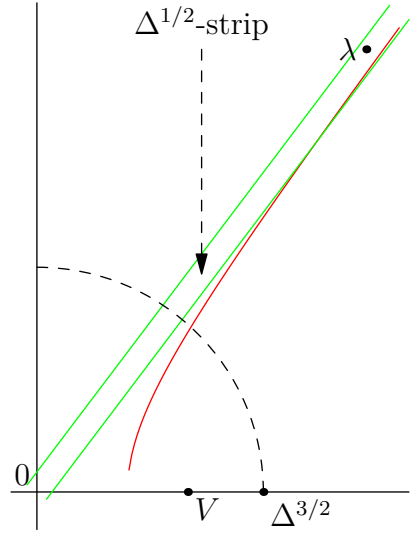

\begin{center}
\begin{asy}
import graph;
import contour;
import markers;
size(7cm);

real    a=2.9, c=5.0;
real    b = sqrt(c^2-a^2);
real    factor=4;

real f(real x) { return sqrt(b*b*(x*x/(a*a)-1)); }
real fline(real x) {return (b/a)*x;}

real t=2.0;
real x=a*cosh(t);
real y=b*sinh(t);
dot("$\lambda$", (x,y), W);

a=3; b = sqrt(c^2-a^2);
draw(graph(f, a+0.05, 1.1*x), red);
pair u=(b/c, -a/c);
draw(shift(0.4*u)*graph(fline, 0, 1.1*x), green);
draw(shift(-0.4*u)*graph(fline, 0, 1.1*x), green);

pair loc1=(x/2, y);
pair loc2=(x/2, fline(x/2));
draw(loc1--loc2, dashed, Arrow);
label("$\Delta^{1/2}$-strip", loc1, N);

draw(arc((0,0), 1.5*c, 0, 90), dashed);
dot("$\Delta^{3/2}$", (1.5*c,0), SE);
xaxis();
dot("$V$", (c,0), SE);
label("0", (0,0), NW);
yaxis();
\end{asy}
\end{center}
\caption{Each $\lambda$ of size $\Abs{\lambda} \ge \Delta^{3/2}$ is in a $O(\Delta^{1/2})$-width strip around an asymptote to a hyperbola $H_k$.}\label{fig:strip-around-hyperbola}
\end{figure}

More specifically we can show that each point $\lambda \in \Lambda$ with $\Abs{\lambda} \ge \Delta^{3/2}$ must be at distance $O(\Delta^{1/2})$ from the asymptote to one of the hyperbolas $H_k$. See Fig.\ \ref{fig:strip-around-hyperbola}.

According to Lemma \ref{cor:strip} each strip whose midline is an asymptote to a hyperbola and its width is $O(\Delta^{1/2})$ may contain at most 2 $\lambda$'s and since there are $O(\Delta)$ such strips we have a total of $O(\Delta)$ points of $\Lambda$ outside the disk centered at $0$ of radius $\Delta^{3/2}$. The size of $\Lambda$ in that disk can be bounded by applying Corollary \ref{cor:smallest-gap} with $R=\Delta^{3/2}$ and we see that this part of $\Lambda$ is also of size $O(\Delta)$. This completes the proof that $\Abs{\Lambda} = O(t)$.

To get the bound on $\Abs{\Lambda \cap [-R, R]^d}$ in terms of $R$ we only need to observe that we have two upper bounds for $\Abs{\Lambda}$: the bound $O(R/\Delta^{1/2})$ coming from Corollary \ref{cor:smallest-gap} and the bound $\Abs{\Lambda} = O(\Delta)$ that we just finished proving. The minimum of these two upper bounds is $O(R^{2/3})$ which is the promised bound.

In a major recent development it has been shown \cite{zakharov2024sets}
$$
\Abs{\Lambda \cap [-R, R]^2} = O(R^{3/5+\epsilon})
$$
for any positive $\epsilon$. The proof in \cite{zakharov2024sets} builds on Corollary \ref{cor:strip} and on Theorem \ref{thm:disk-bound}.

%

\section{Weak tiling} \label{s:w-tiling}


\subsection{Spectrality implies weak tiling}\label{ss:spectrality-implies-weak-tiling}
Suppose we work on a group $G$ with counting measure, such as any finite group or $\ZZ^d$, and $E = \Set{e_1, e_2, \ldots, e_N} \subseteq G$ is a finite spectral set, of size $N$. The spectrum $\Lambda  = \Set{\lambda_1, \lambda_2, \ldots, \lambda_N}\subseteq \ft{G}$ also has $N$ elements. Orthogonality of the characters means that the matrix
$$
M = \left( \lambda_j(e_k) \right)_{j, k = 1, 2, \ldots, N}
$$
has orthogonal rows, so it also has orthogonal columns. This means that $E$, viewed now as a subset of the dual group of $\ft{G}$, which is $G$ itself, is a spectrum of $\Lambda$ (the points $g\in G$ act on the characters $\gamma \in \ft{G}$ by $\gamma \to \gamma(g)$).

By \eqref{spectrum-as-tiling} we now have the tiling conditions
$$
N^{-2} \sum_{\lambda \in \Lambda} \Abs{\ft{\one_E}(t-\lambda)}^2 = 1
\ \ \text{ ($\Lambda$ is a spectrum of $E$) }
$$
and
$$
N^{-2} \sum_{e \in E} \Abs{\ft{\one_\Lambda}(x-e)}^2 = 1
\ \ \text{ ($E$ is a spectrum of $\Lambda$) }.
$$
These can be rewritten as the convolutions
$$
N^{-2} \Abs{\ft{\one_E}}^2 * \one_\Lambda = 1
\ \text{ and }\ 
N^{-2} \Abs{\ft{\one_\Lambda}}^2 * \one_E = 1.
$$
The second of those can be viewed as a tiling by $E$ with a function, or set of fractional translates. Let us write $w(x) = N^{-2} \Abs{\ft{\one_\Lambda}}^2(x)$ for $x \in G$. Then the second convolution above becomes
\begin{equation}\label{w-tiling}
\one_E * w = 1.
\end{equation}
The function $w$ is nonnegative and is equal to $1$ at $0$ of $G$. If the function $w$ happens to take only the values 0 or 1 then \eqref{w-tiling} would be an ordinary tiling of $G$ by translates of $E$ (at the locations where the 1's occur in $w$). In general $w$ does not have this property so we call this situation a \textit{weak tiling} by $E$. So weak tiling is whenever \eqref{w-tiling} holds for some function $w$ on $G$ which is nonnegative and equals $1$ at $0 \in G$.

We have proved:
\begin{theorem}\label{thm:discrete-w-tiling}
If $E \subseteq G$ is a finite set which is spectral with the counting measure then there is a weak tiling of $G$ by $E$.
\end{theorem}

It is one of the greatest developments in the area that Lev and Matolcsi \cite{lev2022fuglede} proved Theorem \ref{thm:discrete-w-tiling} to the case of $\RR^d$ and, in the process, made us realize the importance of this concept.

\begin{theorem}\label{thm:rn-w-tiling}
Suppose $E \subseteq \RR^d$ is a bounded measurable set which is spectral. Then there exists a nonnegative, locally finite measure $\nu$ on $\RR^d$ such that
$$
\one_E * (\delta_0 + \nu) = 1, \text{ almost everywhere in } \RR^d.
$$
\end{theorem}

In both Theorem \ref{thm:discrete-w-tiling} and Theorem \ref{thm:rn-w-tiling} it is crucial that we demand that the measure ($w$ in the discrete case, $\delta_0+\nu$ in the case of $\RR^d$) has a unit point mass at the origin. (Often we state the same condition as the fact that the domain $E$ can tile its complement by fractional copies of itself). Had we not insisted on that condition then every set $E$ would admit a weak tiling (by a multiple of the Haar measure) so admitting a weak tiling would be worthless in distinguishing the spectral sets.

\subsection{Weak tiling does not imply tiling}\label{ss:w-tiling-vs-tiling}
The question if weak tiling by a set $E$ implies that $E$ can also tile is obviously one of the first questions that comes to mind. If the answer were affirmative then we would have immediate confirmation of the ``spectral $\Longrightarrow$ tiling'' direction of the Fuglede Conjecture. Since we know that this direction fails in dimension 3 and higher we can conclude that there are sets $E$ which tile weakly but cannot tile. Any counterexample to the ``spectral $\Longrightarrow$ tiling'' direction is such an example.

We do not know of any such examples when the group is a product of at most two cyclic groups. Another important reason, besides the Fuglede Conjecture, that we would like to know if the implication ``weak tiling $\Longrightarrow$ tiling'' is true in a category of groups, say the cyclic groups, is that if the implication is true we would immediately have a polynomial time (in $N$, the size of the group $G$) algorithm which, given a subset $A \subseteq G$ would decide if $A$ tiles $G$ or not. The algorithm would decide if $A$ weakly tiles $G$ via linear programming, which takes polynomial time \cite{kolountzakis2009algorithms}.

\qq{\label{q:cyclic-group-tiling-algorithm}
Find an algorithm that runs in time polynomial in $N$ which decides if a given $A \subseteq \ZZ_N$ tiles the group $\ZZ_N$ by translations.
}

\subsection{Using the weak-tiling necessary condition to disprove spectrality.}\label{ss:using-w-tiling}

It is hard to overemphasize the importance of this necessary condition for spectrality. Perhaps one example of its use will immediately convey its strength. Of all three domains in Fig.\ \ref{fig:holes} only the one-dimensional one (the union of two intervals) was known \cite{laba2001twointervals} to be non-spectral and the proof of this was somewhat involved. But knowing the necessity of weak tiling for a set to be spectral, it is immediately clear that none of these sets can weakly tile space (fractionally tile its complement) because of the presence of the ``hole'' which cannot be covered with whole or fractional copies of the set without these copies touching the one whole copy that we must have at the origin.

\begin{figure}[h]
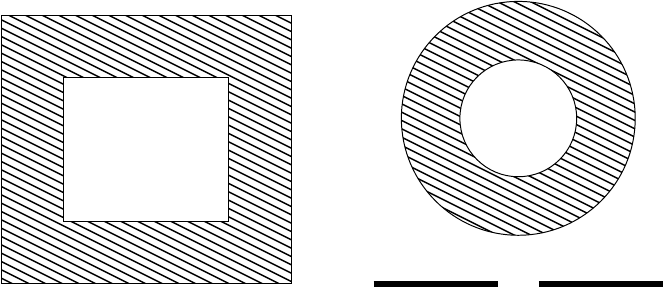
\caption{These domains are not spectral as they cannot weakly tile $\RR^2$ or $\RR$.}\label{fig:holes}
\end{figure}

The importance of this criterion was immediately made plain in \cite{lev2022fuglede} where it was used to complete the missing ``spectral $\Longrightarrow$ tiling'' direction of the Fuglede Conjecture for convex bodies.
\begin{theorem}\label{thm:convex-fuglede}
If $E$ is a convex body in $\RR^d$ which is spectral then it is a polytope and it can tile $\RR^d$ face-to-face by translations along a lattice.

Therefore, the Fuglede Conjecture is true for the class of all convex bodies in $\RR^d$.
\end{theorem}
The Fuglede Conjecture for convex sets had a long history of partial results before its eventual confirmation. To begin with, as explained in \S \ref{ss:convex-tiles-are-spectral}, it was known early on that the ``tiling $\Longrightarrow$ spectral'' direction is true for all convex bodies.

In \cite{kolountzakis2000nonsymmetric} it was proved that only symmetric convex bodies can be spectral (see also \S \ref{ss:nonsym}). In \cite{iosevich2001convexbodies} it was shown that smooth convex bodies with everywhere positive Gaussian curvature are not spectral (see also \S \ref{ss:ball}). In \cite{iosevich2003fuglede} it was shown that the Fuglede Conjecture holds for all planar convex bodies and in \cite{greenfeld2017fuglede} it was proved for all convex \textit{polytopes} in dimension 3. 

We will not show here the details of how the weak tiling necessary condition is used in order to show Theorem \ref{thm:convex-fuglede} but we will try to convey the spirit of the proof in how it forces the domain to be a polytope and forbids curvature. In other words, we will show a new proof that the disk in the plane (this proof is valid in any dimension) is not spectral, by showing that it does not admit weak tilings. 

\begin{figure}[h]
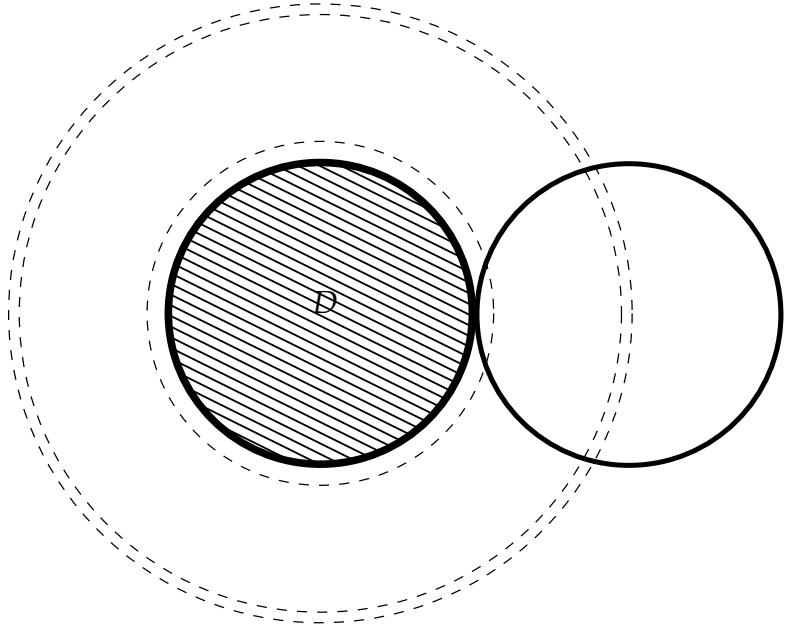
\caption{It is impossible to tile the exterior of a disk $D$ with fractional copies of the same disk. Whenever one tries to fill an area near the boundary of $D$ by placing a copy of $D$ with some nonnegative weight on it the effect is that strictly more weight is placed on the outer ring than on the inner ring.}\label{fig:disks}
\end{figure}

Let us fix the unit disk $D$ in the plane (see Fig.\ \ref{fig:disks}) and assume it is weakly tiling the plane. The shaded disk is the one whole copy of the disk that must be there in the weak tiling. The exterior of $D$ must be covered to a total of level 1 by weighted (fractional) copies of $D$ which are not allowed to overlap the shaded copy of $D$ as it has full weight. Draw two rings concentric with $D$ an inner ring just outside $D$ and an outer ring whose inner radius is 2. Let the width of the inner ring be a small positive number $\epsilon$ and the width of the outer ring be such that the area of the two rings is the same (so the width of the outer ring will be roughly $\epsilon/2$). 

We now claim that whenever we use a copy of $D$ in order to place some weight near $D$ it is always the case that strictly more weight is placed on the inner ring than on the outer ring. Indeed, the worst case for this claim is the one drawn in Fig.\ \ref{fig:disks}, namely when the fractional disk touches $D$. In this case the area covered in the inner ring is roughly $\epsilon^{3/2}$ while the area covered in the outer ring is roughly $\epsilon$, so when $\epsilon$ is sufficiently small the claim is true.
This leads to a contradiction as the total weight that gets placed on each ring must be equal to its area and therefore the two rings should get the same total weight.

\subsection{Fat Cantor sets}\label{ss:fat-cantor-sets}
Let us also see how the weak-tiling necessary condition for spectrality can be used to show that a class of fat Cantor sets (Cantor sets of positive measure) is not spectral \cite{kolountzakis2023spectral}.

Define the set as an infinite intersection of level sets
$$
E = \bigcap_{n=1}^\infty E_n,
$$
where the compact sets  $E_n \subseteq [0, 1]$ are shown in Fig.\ \ref{fig:fat-cantor-set}. We start with $E_0 = [0, 1]$ and at the $n$-th stage we remove an interval of length $d_n$ from the middle of each interval of $E_{n-1}$ thus leaving behind two intervals of length $\ell_n = (\ell_{n-1}-d_n)/2$.

\begin{figure}[h]
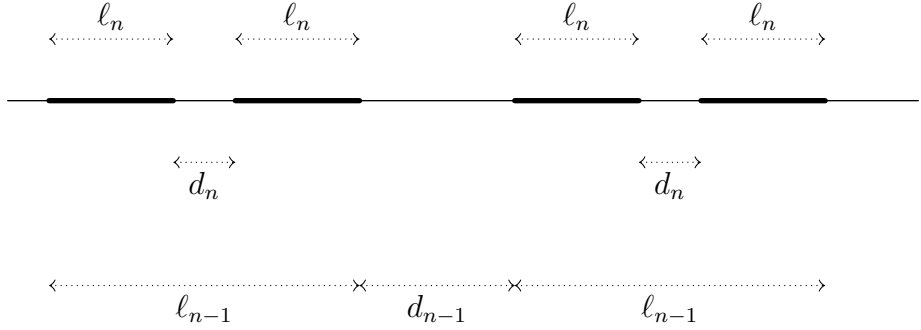

\begin{asy}
size(12cm);

real l=3, h=l/2, d=0.5*l, dd=(2*l+d)*d/l;

draw((-1, 0)--(7l, 0));
Label L1=Label("$\ell_n$", Relative(0.5), N);
Label L2=Label("$d_n$", Relative(0.5), S);
Label L3=Label("$d_{n-1}$", Relative(0.5), S);
Label L4=Label("$\ell_{n-1}$", Relative(0.5), S);

draw((0, 0)--(l, 0), linewidth(2));
draw(L1, (0, h)--(l, h), dotted, Arrows(TeXHead));
draw(L2, (l, -h)--(l+d, -h), dotted, Arrows(TeXHead));

draw((l+d, 0)--(2*l+d, 0), linewidth(2));
draw(L1, (l+d, h)--(2*l+d, h), dotted, Arrows(TeXHead));

draw(L4, (0, -3h)--(2l+d, -3h), dotted, Arrows(TeXHead));

draw(L3, (2*l+d, -3h)--(2*l+d+dd, -3h), dotted, Arrows(TeXHead));

draw((2*l+d+dd, 0)--(2*l+d+dd+l, 0), linewidth(2));
draw(L1, (2*l+d+dd, h)--(2*l+d+dd+l, h), dotted, Arrows(TeXHead));
draw(L2, (2*l+d+dd+l, -h)--(2*l+d+dd+l+d, -h), dotted, Arrows(TeXHead));

draw((2*l+d+dd+l+d, 0)--(2*l+d+dd+l+d+l, 0), linewidth(2));
draw(L1, (2*l+d+dd+l+d, h)--(2*l+d+dd+l+d+l, h), dotted, Arrows(TeXHead));

draw(L4, (2*l+d+dd, -3h)--(2*l+d+dd+l+d+l, -3h), dotted, Arrows(TeXHead));
\end{asy}
\caption{How a fat Cantor set is constructed. The upper row shows 4 intervals of $E_n$ that arose from two intervals of $E_{n-1}$ after we removed a middle interval of legth $d_n$ from each.}\label{fig:fat-cantor-set}
\end{figure}

If we assume that the set $E$ so constructed has positive Lebesgue measure then it follows that we must have $d_n/\ell_n \to 0$.
\begin{theorem}\label{thm:fat-cantor-set-is-not-spectral}
A fat Cantor set as described above cannot weakly tile the real line and is therefore not spectral.
\end{theorem}

It is obvious that each $E_n$ cannot weakly tile since the hole is too small. We have to find a way to pass to the limit. Suppose $\one_E * \mu = 1$ on $\RR$,
$$
\mu = \delta_0+\nu, \text{ and } \nu \ge 0.
$$
Since $E_{n+1} \subseteq E_n$ {\em monotone convergence} gives
$$
\int_{E_n} \one_{E_n} * \nu \to \int_E \one_E * \nu = 0,
$$
from the weak tiling assumption $\one_E*(\delta_0 + \nu) = 1$.

The crucial inequality 
is
$$
\frac{\ell_n-d_n}{d_n} \int_{A_n} \one_{E_n}*\nu \le \int_{E_n}\one_{E_n}*\nu
$$
where $A_n = E_{n-1} \setminus E_n$ (what was thrown out at the $n$-th stage). The intuition behind this inequality is that whenever we are trying to fill the gap $A_n$ using some weighted copies of our set then we end up putting more weight on the remaining set, which will lead to a contradiction, as follows.
\begin{align*}
0 \longleftarrow \int_{E_n} \one_{E_n}*\nu
 &\ge \frac{\ell_n-d_n}{d_n} \int_{A_n} \one_{E_n}*\nu \\
&\ge \frac{\ell_n-d_n}{d_n} \int_{A_n} \one_{E}*\nu \text{\ \  (since $E \subseteq E_n$)} \\
&= \frac{\ell_n-d_n}{d_n} \Abs{A_n} \text{\ \  (due to weak tiling since $A_n \subseteq E^c$)}\\
&= \frac{\ell_n-d_n}{d_n} \frac{d_n}{2\ell_n} \Abs{E_n} \text{\ \  (since $\Ds \frac{\Abs{E_n}}{\ell_n} = 2 \frac{\Abs{A_n}}{d_n})$}  \\
&= \left(\frac12 - \frac{d_n}{2\ell_n}\right) \Abs{E_n}  \\
&\to \frac12 \Abs{E} \text{ as $n\to\infty$.}
\end{align*}
We have reached a contradiction so $E$ cannot weakly tile.

\noindent{\bf Acknowledgement.} The author would like to thank the referee for the extremely thorough work which clarified several obscure points of this paper.

\bibliographystyle{alpha}
\bibliography{mk-bibliography.bib}

\end{document}